\documentclass[12pt]{article}
\usepackage{amsmath,amsthm,latexsym,amsfonts,amssymb,mathrsfs}
\usepackage{amsbsy}
\usepackage{hyperref}
\usepackage[normalem]{ulem} 
\newcommand{\stkout}[1]{\ifmmode\text{\sout{\ensuremath{#1}}}\else\sout{#1}\fi}

\usepackage{tikz}

\def\N{\mathbb{N}}
\newcommand{\calO}{\mathcal{O}}
\newcommand{\calB}{\mathcal{B}}
\newcommand{\calL}{\mathcal{L}}

\newcommand{\bsr}{{\boldsymbol{r}}}

\newcommand{\vf}{{\bf f}}

\newcommand{\bg}{{\bf g}}

\newcommand{\vu}{{\bf u}}
\newcommand{\vv}{{\bf v}}

\newcommand{\bchi}{\boldsymbol \chi}
\newcommand{\bvphi}{\boldsymbol \varphi}
\newcommand{\balpha}{\boldsymbol \alpha}
\newcommand{\bbeta}{\boldsymbol \beta}
\newcommand{\vV}{{\bf V}}
\newcommand{\vw}{{\bf w}}

\newcommand{\bsx}{{\boldsymbol{x}}}
\newcommand{\bsy}{{\boldsymbol{y}}}
\newcommand{\bsz}{{\boldsymbol{z}}}

\newcommand{\vsigma}{\boldsymbol{\sigma}}

\newcommand{\R}{\mathbb{R}}

\newcommand{\calS}{\mathcal{S}}

\numberwithin{equation}{section}

\newcommand{\veps}{{\boldsymbol{\varepsilon}}}
\newcommand{\iprod}[1]{\langle#1 \rangle}

\newtheorem{theorem}{Theorem}[section]

\numberwithin{theorem}{section}
\newcommand{\TheTitle}{Quasi-Monte Carlo sparse grid Galerkin finite element methods for linear elasticity equations with uncertainties}
\newcommand{\TheAuthors}{Dick, Le Gia, Mustapha  and Tran}
\title{{\TheTitle}\thanks{This work was supported by the Australian Research Council
grant DP220101811.}}

\author{M. Clarke, J. Dick, Q. T. Le Gia, K. Mustapha and T. Tran \thanks{School of Mathematics and Statistics, University of New South Wales, Sydney, Australia}}

\ifpdf
\hypersetup{
  pdftitle={\TheTitle},
  pdfauthor={\TheAuthors}}
\fi

\begin{document}
\maketitle
\begin{abstract}
We explore  a linear inhomogeneous elasticity equation with random Lam\'e parameters. The latter  are parameterized by a countably infinite number of terms in separated expansions. The main aim of this work is to estimate expected values (considered as an infinite dimensional integral on the parametric space corresponding to the random coefficients) of linear functionals acting on the solution of the elasticity equation. To achieve this,  the expansions of the random parameters are truncated, a high-order quasi-Monte Carlo (QMC) is combined with a sparse grid approach to approximate the high dimensional integral, and a Galerkin finite element method (FEM) is introduced to approximate the solution of the elasticity equation over the physical domain.  The error estimates from (1) truncating the infinite expansion, (2) the Galerkin FEM, and (3) the QMC sparse grid quadrature rule are all studied. For this purpose, we show certain required regularity properties of the continuous solution with respect to both the parametric and physical variables. To achieve our theoretical regularity and convergence results,  some reasonable assumptions on the expansions of the random coefficients are imposed. Finally, some numerical results are delivered.   
\end{abstract}
\section{Introduction}
In this work we investigate  and analyze the application of  the high-order  quasi-Monte Carlo (QMC) sparse  grid method combined with the conforming  Galerkin finite element  methods (FEMs) to solve a linear elastic model with uncertainties (see \cite{Mathies1997} for an interesting overview of how to incorporate uncertainty into material parameters in linear elasticity problems).  More specifically,  we consider the case  where the properties of the elastic inhomogeneous material  are varying spatially in an uncertain way by using random Lam\'e parameters which are  parametrized by a countably infinite number of parameters. This leads to randomness in both  the Young modulus ($E$) and the Poisson ratio ($\nu$). We intend to measure theoretically the efficiency of our numerical algorithm through the expectation of the random solution over the random field. 

The equation governing small elastic deformations of a body $\Omega$ in~$\R^d$ ($d\in \{2,3\}$) with polyhedral
boundary can be written as
\begin{equation}
 -\nabla \cdot \vsigma(\bsy,\bsz;\vu(\bsx,\bsy,\bsz)) = \vf(\bsx) \quad \text{for } \bsx \in \Omega, \label{eq:L1}
\end{equation}
subject to homogeneous Dirichlet boundary conditions; $\vu(\bsx,\bsy,\bsz) = 0$ for $\bsx \in \Gamma:=\partial \Omega$ and with $\bsy$ and $\bsz$ being parameter vectors describing randomness to be specified later. The parametric Cauchy stress tensor $\vsigma(\bsy,\bsz;\cdot) \in [L^2(\Omega)]^{d\times d}$ is defined as 
\[\vsigma(\bsy,\bsz;\vu(\bsx,\bsy,\bsz)) =\lambda(\bsx,\bsz)\Big(\nabla\cdot \vu(\bsx,\bsy,\bsz)\Big) I 
+ 2\mu(\bsx,\bsy) \veps(\vu(\bsx,\bsy,\bsz)),\] 
with  $\vu(\cdot, \bsy, \bsz)$ being the displacement vector  field of dimension $d$, and the symmetric strain tensor
 $\veps(\vu) := (\nabla \vu + (\nabla \vu)^T)/2\,.$
Here, $\vf  \in [L^2(\Omega)]^d$ is the body force per unit volume and $I$ is the identity tensor.   The gradient ($\nabla$) and the divergence ($\nabla \cdot$) are understood to be with respect to the physical variable $\bsx \in \Omega$.  The Lam\'{e} elasticity parameter $\lambda$ is related to the compressibility of the material, and the shear modulus $\mu$ is related to how the material behaves under normal stress and shear stress, for the material $\Omega$ containing uncertainties. To parametrize these uncertainties, we assume that $\mu = \mu(\bsx,\bsy)$ and $\lambda = \lambda(\bsx,\bsz)$ can be expressed in the following  separate expansions 
\begin{equation}\label{KLexpansion}
 \mu(\bsx,\bsy) = \mu_0(\bsx) +\sum_{j=1}^\infty y_j \psi_j(\bsx)~~{\rm and}~~ 
 \lambda(\bsx,\bsz) = \lambda_0(\bsx) +\sum_{j=1}^\infty z_j \phi_j(\bsx), \quad\bsx\in\Omega,
\end{equation}
where $\{\psi_j\}$ and $\{\phi_j\}$ are orthogonal basis functions for the $L^2(\Omega)$ space. The parameter vectors $\bsy = (y_j)_{j\ge1}$ and $\bsz = (z_k)_{k\ge1}$ belong to $U:=(-\frac{1}{2},\frac{1}{2})^\N$,  consist of a countable number of parameters $y_j$ and $z_k$, respectively,  which are assumed to be i.i.d. uniformly distributed. Using independent random fields for $\lambda$ and $\mu$ in our model assumes that the compressibility and behaviour under stress of the material are, in the range of the parameters of the random $\lambda$ and $\mu$, independent.

The model problem \eqref{eq:L1}  is similar to the one  in \cite{HoangNguyenXia2016, XiaHoang2014}, in which a priori analysis for so-called best $N$-term approximations of standard two-field mixed formulations was investigated. For the well-posedness of \eqref{eq:L1}, we assume that  there are some positive  numbers $\mu_{\min},$ $\mu_{\max}$ and $\lambda_{\max}$  so that 
\begin{equation}\label{ass A2}
0 < \mu_{\min} \le \mu(\bsx,\bsy) \le \mu_{\max}~~{\rm and}~~0 \le \lambda(\bsx,\bsz)\le \lambda_{\max}, \quad \bsx \in \Omega,\;\; \bsy, \bsz \in U.
\tag{A1}
\end{equation}
Due to \eqref{ass A2},  the values of the Poisson ratio of the elastic material $\nu=\frac{\lambda}{2(\lambda+\mu)}$ are ranging between $0$ and $1/2$ which is the case in most materials. Indeed, if $\lambda$ is a constant multiple of $\mu$, that is, the randomnesses  of the Lam\'e parameters are due to the ones in the Young modulus $E$, then $\nu$ is constant. This case was studied in \cite{Khan2019,Khan2021} where the authors introduced a three-field PDE model with a parameter-dependent $E$ which is amenable to discretization by stochastic Galerkin mixed FEMs. The focus in \cite{Khan2019} was on  efficient linear algebra, while an a posteriori error estimation was detailed in \cite{Khan2021}. 
In relation, the authors in \cite{Eigel2014} presented a framework for residual-based a posteriori error estimation and adaptive stochastic Galerkin approximation.  

 If  $\nu$ approaches $1/2$, then we are dealing with an elastic material that becomes nearly incompressible. In this case, and with constant Lam\'e parameters,  the convergence rate of  the conforming piecewise quadratic or cubic Galerkin FEMs is  short by one from being optimal \cite{ScottVogelius1985,Vogelius1983}.
However, the piecewise linear  Galerkin FEM
runs into trouble with the phenomenon of locking where the convergence rates may deteriorate as $\lambda$   becomes too large. This is owing to their inability to represent non-trivial divergence-free displacement fields. 
Locking   can be avoided by using a nonconforming Galerkin FEM  \cite{Vogelius1983, BrennerSung1992,Falk1991}   or by using  mixed FEMs.  These and several other methods were studied very extensively in the existing literature  for the case of constant Lam\'e parameters, we refer the reader to the following books  \cite{Braess2007, BrennerScott2008, BrezziFortin1991}. Investigating the nearly incompressible case with stochastic  Lam\'e parameters   is beyond the scope of this paper; it is a topic of future research. 

Outlines of the paper. The next section is devoted to the statement of the main results of this work.  In Section \ref{VFR}, we derive the variational formulation of the parametric model problem  \eqref{eq:L1}, and prove the existence and uniqueness of the weak solution.  We also investigate certain regularity properties of the continuous solution $\vu$ with respect to the random parameters $\bsy$ and $\bsz$, and the physical parameter $\bsx.$ These results are needed to guarantee the convergence of the errors from both the QMC integration and the Galerkin finite element discretization.  The error from truncating the infinite series expansion in \eqref{KLexpansion} is investigated in Section~\ref{Sec: Truncation}. In  Section \ref{Sec: FEM}, for every $\bsy,\bsz \in U,$ we  approximate the parametric solution $\vu(\cdot,\bsy,\bsz)$ of \eqref{eq:L1} over the physical domain $\Omega$ using the conforming Galerkin FEM,  and discuss the stability and error estimates. In Section \ref{sec: QMC errors}, we investigate the high-order QMC error from  estimating the expected value  of a given function over a high dimensional field.  Firstly, we use a high-order QMC rule for the random coefficients arising from the expansion for $\lambda$ and another QMC rule for the random coefficients arising from the expansion for $\mu$. We study two ways of combining the QMC rules, one is a tensor product structure (Theorem~\ref{prop:qmc}) and the other one is a sparse grid combination (Theorem~\ref{qmc sparse grid}). Secondly,  we use one family of high-order QMC rule to simulate both $\lambda$ and $\mu$ simultaneously (Theorem \ref{prop:qmc6.3}). Designing such a QMC rule might be more complicated, however, it leads to a better rate of convergence.  We end the paper with some numerical simulations in Section \ref{Sec: Numeric}. In a sample of four different examples, we illustrate numerically the achieved theoretical finite element QMC convergence results.  

\section{Main results}
We start this section by introducing  the following vector function  spaces and the associated norms, which we will be using throughout the remainder of the paper.   Let   $\vV:=[H^1_0(\Omega)]^d$ and the associated norm be   $\|\vw \|_{\vV}:=\Big(\sum_{i=1}^d \|w_i\|_{H^1(\Omega)}^2 \Big)^{1/2}$  with the $w_i$'s being the components of the vector function $\vw.$  For $J=0,1,2,\cdots,$ the norm on the vector Sobolev space ${{\bf H}^J}:=[H^J(\Omega)]^d$, denoted by $\|\cdot\|_{{{\bf H}^J}},$ is defined in a similar fashion with $\|w_i\|_{H^J(\Omega)}$ in place of $\|w_i\|_{H^1(\Omega)}.$ We dropout ${\bf H}^J$ from the norm notation on the space ${\bf H}^0={\bf L}^2(\Omega):=[L^2(\Omega)]^d$ for $d \ge 1.$ Here,   $H^1_0(\Omega)$ and $H^J(\Omega)$ are the usual Sobolev spaces. Finally,  $\vV^*$ is denoted  the dual space of $\vV$ with respect to the ${\bf L}^2(\Omega)$  inner product, with norm denoted by $\|\cdot\|_{\vV^\ast}$. 

As mentioned earlier, and more precisely, we are interested in efficient  approximation of the expected value of the function $\calL(\vu(\cdot,\bsy,\bsz)),$ for a certain linear functional $\calL:{\bf L}^2(\Omega)  \to \R$, with respect to the random variables $\bsy$ and $\bsz,$ where $\vu$ is the solution of \eqref{eq:L1}. In other words,  we seek to approximate 
\begin{equation}\label{EV}
\Xi_\vu:=\int_{U}\int_{U}\calL\big( \vu(\cdot,\bsy,\bsz)\big)\,d\bsy\,d\bsz\,,
\end{equation}
where $d\bsy$ and $d\bsz$ are the uniform probability measures on $U$. As a practical  example, we may choose $\calL$  to be a local  continuous average on some domain $\Omega_0 \subset \Omega.$   

To accomplish the above task, and for the practical implementation, the occurred infinite sums in \eqref{KLexpansion} must be truncated. Then,   we approximate  $\vu$ by $\vu_{\bf s}$ which is the solution of \eqref{eq:L1} obtained by truncating the  infinite expansions in \eqref{KLexpansion} where  ${\bf s}=(s_1,s_2)$, that is, assuming that $y_j=z_k=0$ for $j >s_1$ and $k>s_2$. Then, with    $U_i=[0,1]^{s_i}$ for $i=1,2,$ being of (finite) fixed dimension $s_i$,  we estimate the expected value of $\calL(\vu(\cdot,\bsy,\bsz))$ by  approximating 
\begin{equation}\label{F finite}
\Xi_{{\bf s},\vu_{\bf s}}:=\int_{U_2}\int_{U_1}\calL\Big({\vu_{{\bf s}}}\Big(\cdot,\bsy-{\bf \frac{1}{2}},\bsz-{\bf \frac{1}{2}}\Big)\Big)d\bsy\,d\bsz\,.
\end{equation}
In the above finite dimensional integral,  $d\bsy$ and $d\bsz$ are the uniform probability measures on $U_1$ and $U_2$, respectively.  The shifting of the coordinates by ${\bf \frac{1}{2}}$ translates $U_i$ to  $\big[-\frac{1}{2},\frac{1}{2}\big]^{s_i}$ for $i=1,2.$ We approximate such $(s_1+s_2)$-dimensional integrals using a high-order QMC quadrature. Preceding this, we intend to solve the truncated problem over the physical domain $\Omega$ numerically via a continuous  Galerkin  FEM. So,  for every $\bsy \in U_1$ and $\bsz \in U_2,$ we approximate  the truncated solution  $\vu_{\bf s}\Big(\cdot,\bsy-{\bf \frac{1}{2}},\bsz-{\bf \frac{1}{2}}\Big)$  by the parametric spatial Galerkin finite element solution  $\vu_{{\bf s}_h}\Big(\cdot,\bsy-{\bf \frac{1}{2}},\bsz-{\bf \frac{1}{2}}\Big) \in \vV_h\subset \vV$ (see Section \ref{Sec: FEM} for the definition of the finite element space $\vV_h$)  with the sums in \eqref{KLexpansion} truncated to $s_1$ and $s_2$ terms, respectively. In the third step, we  estimate the expectation of the approximation using first a tensor product of two high-order QMC methods. In summary, we approximate the expected value in \eqref{EV} by the following truncated QMC Galerkin finite element rule 
\begin{equation}\label{eq:QMCInt F}
\Xi_{\vu_{{\bf s}_h}, Q}:=\frac{1}{N_1\,N_2}\sum_{j=0}^{N_1-1} \sum_{k=0}^{N_2-1}\calL\Big(\vu_{{\bf s}_h}\Big(\cdot,\bsy_j-{\bf \frac{1}{2}},\bsz_k-{\bf \frac{1}{2}}\Big)\Big)\,,
\end{equation}
where the QMC points $\{\bsy_0, \ldots,\bsy_{N_1-1} \} \in U_1$ and $\{\bsz_0, \ldots,\bsz_{N_2-1} \} \in U_2$. Noting that, as a better alternative of the above QMC tensor product rule,  we discuss an efficient  high-order QMC sparse grid combination (see Theorem \ref{qmc sparse grid}) and also a direct high-order QMC sparse grid rule in $(s_1+s_2)$ dimension (see Theorem \ref{prop:qmc6.3}).    

We have  three sources of error: a dimension truncation error depending on $s_1$ and $s_2$, a Galerkin discretization error depending on the  maximum finite element mesh diameter $h$ of the domain $\Omega$, and a QMC quadrature error which depends on $N_1$ and $N_2$. We split the error as:
\begin{equation}\label{combine}
|\Xi_\vu-\Xi_{\vu_{{\bf s}_h}, Q}|\le |\Xi_\vu-\Xi_{{\bf s}, \vu_{\bf s}}|+|\Xi_{{\bf s},\vu_{\bf s}}-\Xi_{{\bf s},\vu_{{\bf s}_h}}|+|\Xi_{{\bf s},\vu_{{\bf s}_h}}-\Xi_{\vu_{{\bf s}_h}, Q}|.
\end{equation}
Since $d\bsy$ and $d\bsz$ are the uniform probability measures with  i.i.d. uniformly distributed parameters on $U$, 
\[\Xi_\vu-\Xi_{{\bf s}, \vu_{\bf s}}
=\int_{U}\int_{U}\calL\Big(\vu(\cdot,\bsy,\bsz)-\vu_{{\bf s}}(\cdot,\bsy^{s_1},\bsz^{s_2})  \Big)\,d\bsy\,d\bsz,\]
where $\bsy = (y_j)_{j\ge1}, \bsz = (z_k)_{k\ge1} \in U$, and the truncated vectors  $\bsy^{s_1}$ and  $\bsz^{s_2}$ are $(y_1,y_2,\cdots,y_{s_1},0,0,\cdots)$ and $(z_1,z_2,\cdots,z_{s_2},0,0,\cdots)$, respectively. To estimate this term, we refer to the dimension truncation error which is analyzed in Theorem \ref{Truncating error}.  To reduce the errors from such a truncation, which is necessary from a practical point of view, we assume that the $L^2(\Omega)$ orthogonal basis functions  $\psi_j$ and $\phi_j$  are ordered
so that $\|\psi_j\|_{L^\infty(\Omega)}$ and $\|\phi_j\|_{L^\infty(\Omega)}$ are  nonincreasing. That is, 
\begin{equation}\label{ass A5}
\| \psi_j \|_{L^\infty(\Omega)} \ge
\| \psi_{j+1} \|_{L^\infty(\Omega)} ~~{\rm and}~~ \| \phi_j \|_{L^\infty(\Omega)} \ge
\| \phi_{j+1} \|_{L^\infty(\Omega)},\quad {\rm for}~~j\ge 1\,.
\tag{A2}
\end{equation}
For the convergence from the series truncation,  we assume that   $\mu_0,\,\lambda_0  \in L^\infty(\Omega)$, and 
  \begin{equation}\label{ass A1}
 \sum_{j=1}^\infty \|\psi_j\|^p_{L^\infty(\Omega)} < \infty~~{\rm and}~~
 \sum_{j=1}^\infty \|\phi_j\|^q_{L^\infty(\Omega)} < \infty,\quad{\rm for~some}~~0<p,\,q\le 1\,.
\tag{A3}
\end{equation}
When $p=1$ and$/$or $q=1,$ it is essential to have    
  \begin{equation}\label{ass A3}
 \sum_{j=s_1+1}^\infty \|\psi_j\|_{L^\infty(\Omega)} \le  C s_1^{1-1/\varrho_1}~~{\rm and/or} 
 \sum_{j=s_2+1}^\infty \|\phi_j\|_{L^\infty(\Omega)} \le  Cs_2^{1-1/\varrho_2},
\tag{A4}
\end{equation}
for some $0<\varrho_1,\varrho_2<1$.  The second term on the right-hand side of  \eqref{combine} is the finite dimensional integral of the  linear functional $\calL$ acting on the difference between the truncated solution $\vu_{\bf s}$ and its approximation $\vu_{{\bf s}_h}$. This can be deduced from Theorem \ref{Convergence theorem} by replacing the vectors $\bsy$ and $\bsz$ with $\bsy^{s_1}$ and $\bsz^{s_2}$, respectively,  and using the fact that $\vu_{\bf s}$ satisfies the regularity properties in Theorem \ref{lem: vu bound}. The Lam\'e parameters $\mu(\cdot,\bsy)$ and $\lambda(\cdot,\bsz)$ are required to be in the Sobolev space $W^{\theta,\infty}(\Omega)$ for every $\bsy,\,\bsz  \in U$ and for some integer $1\le \theta\le r$ ($r$ is the degree of the finite element solution, so $\theta=1$ in the case of piecewise linear Galerkin FEM). To have this, we assume that 
\begin{equation}\label{ass A4}
\mu_0,\,\lambda_0  \in W^{\theta,\infty}(\Omega),~~ \sum_{j=1}^\infty \|\psi_j\|_{W^{\theta,\infty}(\Omega)}~~ {\rm and}~~ \sum_{j=1}^\infty \| \phi_j\|_{W^{\theta,\infty}(\Omega)}~{\rm are~finite}\,.
\tag{A5}
\end{equation}
It is clear from Theorem \ref{Convergence theorem} that for every $(\bsy,\bsz) \in U_1\times U_2$, $|\calL(\vu_{\bf s}-\vu_{{\bf s}_h})|$ converges faster than  $\|\vu_{\bf s}-\vu_{{\bf s}_h}\|_\vV$ provided that the linear functional  $\calL$ is  bounded in the $L^2(\Omega)$ sense (that is, $|\calL(\vw)|\le \|\calL\|\|\vw\|$ for any $\vw \in {\bf L}^2(\Omega)$), which is not always guaranteed. 
  
The third term in \eqref{combine} is the QMC quadrature error which can be estimated by applying Theorem \ref{prop:qmc} or Theorem~\ref{qmc sparse grid} or Theorem \ref{prop:qmc6.3} with ${F}(\bsy,\bsz):=\calL\Big(\vu_{{\bf s}_h}\Big(\cdot,\bsy-{\bf \frac{1}{2}},\bsz-{\bf \frac{1}{2}}\Big)\Big).$ Noting that (the mixed derivative)  $|\partial_{\bsy,\bsz}^{\balpha,\bbeta}{F}(\bsy,\bsz)|\le \|\calL\|_{\vV*} \,\|\partial_{\bsy,\bsz}^{\balpha,\bbeta}\vu_{{\bf s}_h}\big(\cdot,\bsy-{\bf \frac12},\bsz-{\bf \frac12}\big)\|_{\vV}$, and then, Theorem \ref{lem: vu bound y z} can be applied to verify the regularity conditions in \eqref{eq:like-norm} 
and \eqref{eq:like-norm mixed}    which are necessary for the QMC (full and sparse grid) error results (Theorems \ref{prop:qmc} and \ref{qmc sparse grid}) and the $(s_1+s_2)$-dimensional QMC error results (Theorem \ref{prop:qmc6.3}), respectively. Here, Assumptions \eqref{ass A5} and \eqref{ass A1} (for $p=q=1$ only) are needed.

We summarize the combined error estimate in the next theorem. In addition to Assumptions \eqref{ass A2}--\eqref{ass A4}, we assume that the physical domain $\Omega$  is  ${\mathcal C}^{\theta,1}$ or the boundary of $\Omega$ is of class ${\mathcal C}^{\theta+1}$ for some integer $\theta\ge 1$ (for $\theta=1$, $\Omega$ can be convex instead)  and the body force  vector function $\vf$ belongs to ${\bf H}^{\theta-1}(\Omega)$ (recall that ${\bf H}^0(\Omega)={\bf L}^2(\Omega)$). These additional assumptions are needed to guarantee  that the strong solution $\vu$ of \eqref{eq:L1}  is in the space ${\bf H}^{\theta+1}(\Omega)$, see Theorem \ref{lem: vu bound}. This property is essential for the optimal convergence of the finite element solution of \eqref{eq:L1} over $\Omega.$   Throughout the paper, $C$ is a generic constant that is independent of  $h$, the number of QMC points $N_i$, and the dimension $s_i$ for $i=1,\,2,$   but may depend on the physical domain $\Omega$ and other parameters that will be mentioned accordingly. 
\begin{theorem}\label{main results}
Let $\vu$ be the solution of problem \eqref{eq:L1} and $\vu_{{\bf s}_h}$ be the Galerkin finite element solution of degree $\le r$ (with $r\ge 1$)  defined as in \eqref{FE solution} with $y_j=z_k=0$ for $j>s_1\ge 1$ and $k>s_2\ge 1.$  For  $i=1,\,2$,  let  $N_i = b^{m_i}$ with $m_i$ being positive integers and $b$ being prime. Then one can construct two interlaced polynomial lattice rules of order $\alpha \,:=\, \lfloor 1/p \rfloor +1$ with $N_1$ points, and of order $\beta:=\, \lfloor 1/q \rfloor +1$ with $N_2$ points where $|m_1 q - m_2 p| < 1$, and $p$ and $q$ are those in \eqref{ass A1}, such that the following QMC Galerkin finite element error bound holds: for $1\le \theta\le r,$
\[
    |\Xi_\vu-\Xi_{\vu_{{\bf s}_h}, Q}|\le C \, h^{\theta+1} \|\vf\|_{{\bf H}^{\theta-1}}  \|\calL\|
  +  C\,\Big(s_1^{1-\max(1/p,1/\varrho_1)}+ s_2^{1-\max(1/q,1/\varrho_2)}+ N^{-\frac{1}{p+q}} \Big)\|\vf\|_{\vV^*} \|\calL\|_{\vV^*}\,,
\]
where $N = N_1 N_2$ is the total number of QMC quadrature points. The constant $C$ depends on $b,p,q,\lambda,$ and $\mu$, but is independent of $s_i$ and $m_i$ for $i\in \{1,2\},$ and  $h$.

If we use a QMC rule in dimension $s_1 + s_2$ directly with $N$ points (see Theorem~\ref{prop:qmc6.3}), then the above estimate remains valid with  $N^{-\min(1/p, 1/q)}$ in place of $N^{-\frac{1}{p+q}}$. Further, there exists a combined QMC sparse grid approximation given by \eqref{eq: truncate infinite sums} (see Theorem~\ref{qmc sparse grid}) such that the above error bound remains valid with $M^{-\min(1/p, 1/q)/2}$ in place of $N^{-\frac{1}{p+q}}$, where $M$ is the total number of quadrature points in the QMC sparse grid approximation. The constant $C$ in the new bound  depends on $b,p,q,\lambda,$ and $\mu$, but is independent of $s_1, s_2$, $M$, and $h$.
\end{theorem}

\section{Weak  formulation and regularity}\label{VFR}
This section is devoted to deriving the weak formulation of the parametric elasticity equation  \eqref{eq:L1} for each value of the parameter $\bsy,\,\bsz \in U.$ Then we show some useful regularity properties of the weak solution with respect to both the  physical variable $\bsx$ and parametric variables $\bsy$ and $\bsz$. Preceding  this, we establish the existence and uniqueness of the weak solution. 

For the weak formulation of \eqref{eq:L1}, for every  $\bsy,\bsz \in U$, we multiply  both sides of \eqref{eq:L1} by a test function $\vv \in \vV$, and use Green’s formula and the given homogeneous Dirichlet boundary conditions after integrating over the physical domain $\Omega.$  Then, the usage of the identities 
\[\vsigma(\bsy,\bsz;\vu):\nabla \vv = \vsigma(\bsy,\bsz;\vu): \veps(\vv)=\lambda \nabla \cdot \vu \nabla \cdot \vv +2\mu \veps(\vu):\veps(\vv)\] (the colon operator is the inner product between tensors) results in the following parameter-dependent
weak formulation:  find $\vu(\cdot,\bsy,\bsz) \in \vV$ satisfying
\begin{equation}\label{para weak}
 \calB(\bsy,\bsz;\vu, \vv) = \ell(\vv), \quad \text{for all} \quad \vv \in \vV,\quad{\rm for~every}~~~\bsy,\,\bsz \in U,
\end{equation}
where  the bilinear form $\calB(\bsy,\bsz;\cdot,\cdot)$ and the linear functional $\ell(\cdot)$ are defined by 
\begin{equation}\label{eq: bilinear}
  \calB(\bsy,\bsz;\vu, \vv) := \int_\Omega [2\mu\, \veps(\vu):\veps(\vv)+\lambda \nabla \cdot \vu \nabla \cdot \vv] \,d\bsx~{\rm and}~
  \ell(\vv) := \iprod{\vf,\vv}:=\int_\Omega \vf \cdot \vv \,d\bsx.
\end{equation}

Next, we show the existence and uniqueness of the  solution of  \eqref{para weak}. Noting that, for a given vector function $\vv$,  
$\|\veps(\vv)\|=\Big(\int_\Omega \veps(\vv):\veps(\vv)\,d\bsx\Big)^{1/2}$ and similarly, the norm $\|\nabla \vv\|$ is defined.
\begin{theorem}\label{thm: unique solution}
Assume that \eqref{ass A2}
is satisfied. Then, for every $f \in \vV^*$ and $\bsy,\,\bsz  \in U$, the parametric weak formulation problem \eqref{para weak} has a unique solution. 
\end{theorem}
\begin{proof}
Using  assumption \eqref{ass A2} and applying  the Cauchy-Schwarz inequality leads to  
\[|\calB(\bsy,\bsz;\vv, \vw)| \le 2 \mu_{\max}\|\veps(\vv)\|\,\|\veps(\vw)\|+\lambda_{\max} \|\nabla \cdot \vv\|\,\| \nabla \cdot \vw\|\,.\]
Hence, using the inequalities $\|\nabla \cdot \vw\| \le  \sqrt{d}\,\|\nabla \vw\|$ and $\|\veps(\vw)\| \le  
\|\nabla \vw\|$, we have
\begin{equation}\label{eq: bounded}
|\calB(\bsy,\bsz;\vv,\vw)| \le (d \lambda_{\max}+2 \mu_{\max})\|\nabla \vv\|\,\|\nabla \vw\|\le 
(d \lambda_{\max}+2 \mu_{\max})\| \vv\|_{\vV}\,\|\vw\|_{\vV},
\end{equation}
for any $\vv,\vw \in \vV$. So, the bilinear form $\calB(\bsy,\bsz;\cdot,\cdot)$ is bounded on $\vV \times \vV$. For the coercivity property of $\calB(\bsy,\bsz;\cdot,\cdot)$, we use again assumption \eqref{ass A2} in addition to  Korn's inequality to obtain 
\begin{equation}\label{eq: coer B}
 \calB(\bsy,\bsz;\vv,\vv)\ge  2 \mu_{\min} \|\veps(\vv)\|^2 \ge C \mu_{\min} \| \vv\|_{\vV}^2, \quad \vv \in \vV.
\end{equation}
Since $\ell(\cdot)$  is a bounded linear functional on $\vV$,  an application of the Lax-Milgram theorem completes the proof. 
\end{proof}

For the finite element error analysis, we discuss next some required regularity properties of the parametric solution of \eqref{para weak}. For the nearly incompressible case (which is beyond the scope of this work), one has to be more specific about the constant  $\widetilde C$ in the following theorem.    
\begin{theorem}\label{lem: vu bound}
Assume that \eqref{ass A2} is satisfied. Then, for every $\vf \in \vV^*$ and every $\bsy,\,\bsz  \in U$, the parametric weak  solution $\vu=\vu(\cdot,\bsy,\bsz)$ of  problem \eqref{para weak} satisfies
\begin{equation}\label{a priori}
  \|\vu(\cdot,\bsy,\bsz)\|_{\vV}  \le C\mu_{\min}^{-1} \|\vf\|_{\vV^*}\,.
\end{equation}
If  $\Omega$ is ${\mathcal C}^{\theta,1}$ (or the boundary of $\Omega$ is of class ${\mathcal C}^{\theta+1}$) for some integer $\theta\ge 1$ (for $\theta=1$, we may assume that $\Omega$ is convex instead),  then  $\vu(\cdot,\bsy,\bsz) \in \vV \cap {\bf H}^{\theta+1}(\Omega)$ (that is, it is a strong solution of problem \eqref{eq:L1}) provided that \eqref{ass A4} is satisfied and  $\vf \in {\bf H}^{\theta-1}(\Omega)$. Furthermore, 
\begin{equation}\label{a priori H2}
  \|\vu(\cdot,\bsy,\bsz)\|_{{\bf H}^{\theta+1}}
 \le \widetilde C\,\|\vf\|_{{\bf H}^{\theta-1}},\quad{\rm  for~ every}~~ \bsy,\,\bsz  \in U \,.
\end{equation} 
The constant $\widetilde C$   depends on $\Omega$, $\mu$,  $\lambda$, including $\|\mu(\cdot,\bsy)\|_{W^{\theta,\infty}(\Omega)}$ and $\|\lambda(\cdot,\bsz)\|_{W^{\theta,\infty}(\Omega)}$.
\end{theorem}
\begin{proof}
From the coercivity property in \eqref{eq: coer B} and the weak formulation in \eqref{para weak}, we have  
\[
C \mu_{\min}\| \vu\|_{\vV}^2\le  2 \mu_{\min}  \|\veps(\vu)\|^2 \le   \calB(\bsy,\bsz;\vu,\vu)=\ell(\vu) \le \|\vf\|_{\vV^*} \|\vu\|_{\vV},\]
for every $\bsy,\,\bsz  \in U$. Thus, the proof of the regularity estimate in \eqref{a priori} is completed.

For every $\bsy,\,\bsz  \in U$, the operator $\nabla \cdot \vsigma$ in the elasticity equation \eqref{eq:L1} is strongly elliptic because the bilinear operator $\calB(\bsy,\bsz;\cdot,\cdot)$ is coercive on $\vV$. 
Thus, due to the imposed assumptions on $\Omega$, the solution $\vu(\cdot,\bsy,\bsz)$ of \eqref{para weak} is in the space ${\bf H}^{\theta+1}$ and satisfies the regularity property in \eqref{a priori H2}. See \cite[Theorem 6.3-6]{Cialert1988}, \cite[Theorems  2.4.2.5, 2.5.1.1, and 3.2.1.3]{Grisvard1985}  and \cite[Chapter 4]{McLean2000} for more details about strongly elliptic system and elliptic regularity. 
\end{proof}

In the QMC (see Theorem \ref{prop:qmc}) and QMC sparse grids (see Theorem \ref{qmc sparse grid}) error estimations, we need to bound the mixed first partial derivatives of the parametric displacement  $\vu$ with respect to the random variables $y_j$ and $z_k$. This will be the topic of the next theorem. For convenience,  we introduce  ${\calS}$ to be the set of (multi-index) infinite vectors  $\balpha=(\alpha_j)_{j\ge 1}$ with nonnegative integer entries such that  $|\balpha|:=\sum_{j\ge 1} \alpha_j<\infty.$ That is,  sequences of nonnegative integers for which only finitely many entries are nonzero. For  $\balpha=(\alpha_j)_{j\ge 1}$  and $\bbeta=(\beta_j)_{j\ge 1}$ belonging to  $\calS,$ the mixed partial derivative $\partial_{\bsy,\bsz}^{\balpha,\bbeta}$ is defined by   
\[\partial_{\bsy,\bsz}^{\balpha,\bbeta}:=\partial_{\bsy}^{\balpha}\partial_{\bsz}^{\bbeta}   =\frac{\partial^{|\balpha|}}{\partial_{y_1}^{\alpha_1}\partial_{y_2}^{\alpha_2}\cdots}\frac{\partial^{|\bbeta|}}{\partial_{z_1}^{\beta_1}\partial_{z_2}^{\beta_2}\cdots}\,. 
\]
It reduces to $\partial_{\bsy}^{\bbeta}$ and  $\partial_{\bsz}^{\balpha}$ when $|\bbeta|=0$ and $|\balpha|=0,$ respectively.  
\begin{theorem}\label{lem: vu bound y z}
Assume that  \eqref{ass A5} and \eqref{ass A1} (for $p=q=1$) are satisfied. Then,  for  $\vf \in \vV^*$, $\bsy,\bsz \in  U$, and $\balpha,\bbeta \in \calS$, the solution $\vu(\cdot,\bsy,\bsz)$ of the parametric weak problem \eqref{para weak} satisfies
\begin{equation}\label{mixed estimate}
\big\|\veps\big(\partial_{\bsy,\bsz}^{\balpha,\bbeta} \vu(\cdot,\bsy,\bsz)\big)\big\| \le 
(|\balpha|+|\bbeta|)! \
\widetilde {\bf b}^{\balpha} \widehat {\bf b}^{\bbeta}\|\veps(\vu)\|,
\end{equation} 
where 
\[\widetilde {\bf b}^{\balpha}=\prod_{i\ge 1} (\widetilde b_i)^{\alpha_i}~~{\rm and}~~ \widehat {\bf b}^{\bbeta}=\prod_{i\ge 1} (\widehat b_i)^{\beta_i},~~{\rm with}~~
 \widetilde b_j=\frac{\|\psi_j\|_{L^\infty(\Omega)}}{\mu_{\min}}~~{\rm and}~~ \widehat b_j=\frac{d}{2}\,\frac{\|\phi_j\|_{L^\infty(\Omega)}}{\mu_{\min}}\,.\]
  Consequently,
 \begin{equation}\label{mixed estimate 2}
 \|\partial_{\bsy,\bsz}^{\balpha,\bbeta} \vu(\cdot,\bsy,\bsz)\|_{\vV} \le 
C\mu_{\min}^{-1}(|\balpha|+|\bbeta|)! \
\widetilde {\bf b}^{\balpha} \widehat {\bf b}^{\bbeta}\|\vf\|_{\vV^*},
\end{equation}
where the constant $C$ depends on $\Omega$  only.
\end{theorem}
\begin{proof} Differentiating both sides of \eqref{para weak} with respect to the variables $y_j$ and $z_k$, we find the following recurrence after a tedious calculation 
\begin{equation}\label{recurrence}
 \calB(\bsy,\bsz;\partial_{\bsy,\bsz}^{\balpha,\bbeta}\vu, \vv) 
  =-2\sum_{\balpha}\alpha_j \int_\Omega \psi_j \veps(\partial_{\bsy,\bsz}^{\balpha-{\bf e}_j,\bbeta} \vu):\veps(\vv)\,d\bsx
  -\sum_{\bbeta}\beta_k \int_\Omega \phi_k \nabla\cdot(\partial_{\bsy,\bsz}^{\balpha,\bbeta-{\bf e}_k} \vu)\nabla\cdot\vv\,d\bsx,\end{equation}
for all $\vv \in \vV,$ where $\sum_{\balpha}=\sum_{j,\alpha_j\ne 0}$ (that is, the sum over the nonzero indices of $\balpha$), and   ${\bf e}_i \in \calS$ denotes the multi-index with entry $1$ in position $i$ and zeros elsewhere. 

Choosing $\vv=\partial_{\bsy,\bsz}^{\balpha,\bbeta}\vu$ in \eqref{recurrence}, then using the inequality $\|\nabla \cdot \vv\|\le \sqrt{d}\,\|\veps(\vv)\|$ and  the coercivity property of $\calB(\bsy,\bsz;\cdot,\cdot)$ in \eqref{eq: coer B}, after some simplifications, we obtain    
\[ \|\veps(\partial_{\bsy,\bsz}^{\balpha,\bbeta} \vu)\|^2
\le 
\sum_{\balpha}\alpha_j \widetilde b_j \|\veps(\partial_{\bsy,\bsz}^{\balpha-{\bf e}_j,\bbeta} \vu)\|\,
\|\veps(\partial_{\bsy,\bsz}^{\balpha,\bbeta} \vu)\|
+ \sum_{\bbeta}\beta_k\, \widehat b_k
\|\veps(\partial_{\bsy,\bsz}^{\balpha,\bbeta-{\bf e}_k} \vu)\|\,\|\veps(\partial_{\bsy,\bsz}^{\balpha,\bbeta} \vu)\|\,,\]
and consequently, 
\begin{equation}\label{estimate 1}
    \|\veps(\partial_{\bsy,\bsz}^{\balpha,\bbeta} \vu)\|
\le 
\sum_{\balpha}\alpha_j \widetilde b_j \|\veps(\partial_{\bsy,\bsz}^{\balpha-{\bf e}_j,\bbeta} \vu)\|\,
+ \sum_{\bbeta}\beta_k\, \widehat b_k
\|\veps(\partial_{\bsy,\bsz}^{\balpha,\bbeta-{\bf e}_k} \vu)\|\,.
\end{equation}
When $\bbeta={\bf 0},$ the above inequality reduces to 
\begin{equation}\label{estimate 2}
     \|\veps(\partial_{\bsy}^{\balpha} \vu)\| \le 
\sum_{\balpha}\alpha_j \widetilde b_j \|\veps(\partial_{\bsy}^{\balpha-{\bf e}_j} \vu)\|\,.
\end{equation}
Recursively, we obtain 
\[    \|\veps(\partial_{\bsy}^{\balpha} \vu)\| \le 
 [\,|\balpha| \,(|\balpha|-1)\,\cdots 1]\,\prod_{i\ge 1} (\widetilde b_i)^{\alpha_i}\|\veps(\vu)\|= |\balpha|!
\prod_{i\ge 1} (\widetilde b_i)^{\alpha_i}\|\veps(\vu)\|,
\]
and hence,  \eqref{mixed estimate} holds true in this case. Similarly, we can show \eqref{mixed estimate} when $\balpha={\bf 0}.$ 

When $\balpha$ and $\bbeta$ are both not identically zero vectors, the above approach can be extended, however it is not easy to follow. Owing to this, following \cite{CohenDeVoreSchwab2010}, we use instead the induction hypothesis on $n:=|\balpha+\bbeta|$. From the above contribution, it is clear that \eqref{mixed estimate} holds true when $|\balpha+\bbeta|=1$. Now, assume that \eqref{mixed estimate} is true for $|\balpha+\bbeta|=n$, and the task is to claim  it for $|\balpha+\bbeta|=n+1.$  From \eqref{estimate 1} and the induction hypothesis, we have 
\begin{multline*}
    \|\veps(\partial_{\bsy,\bsz}^{\balpha,\bbeta} \vu)\|\le 
n! \Big(\sum_{\balpha}\alpha_j \widetilde b_j  
\widetilde {\bf b}^{\balpha-{\bf e}_j} \widehat {\bf b}^{\bbeta}+ \sum_{\bbeta}\beta_k\, \widehat b_k
\widetilde {\bf b}^{\balpha} \widehat {\bf b}^{\bbeta-{\bf e}_k}\Big)\|\veps(\vu)\|\\
= n!\widetilde {\bf b}^{\balpha} \widehat {\bf b}^{\bbeta}\Big(\sum_{\balpha}\alpha_j  + \sum_{\bbeta}\beta_k\Big)\|\veps(\vu)\|\,.
\end{multline*}
Since $\sum_{\balpha}\alpha_j  + \sum_{\bbeta}\beta_k=n+1$, the proof of \eqref{mixed estimate} is completed.  

Finally, since $\|\veps(\partial_{\bsy,\bsz}^{\balpha,\bbeta} \vu(\cdot,\bsy,\bsz))\| \ge C\,\|\partial_{\bsy,\bsz}^{\balpha,\bbeta} \vu(\cdot,\bsy,\bsz)\|_{\vV}$ (by Korn's inequality) and since $\|\veps(\vu)\|\le \|\nabla \vu\|\le C\mu_{\min}^{-1}\|\vf\|_{\vV^*}$ (by \eqref{a priori}), we derive \eqref{mixed estimate 2} from \eqref{mixed estimate}. 
 \end{proof}
\section{A truncated problem}\label{Sec: Truncation}
This section is dedicated to investigating the error from truncating the first and second sums in \eqref{KLexpansion} at $s_1$ and $s_2$ terms, respectively, from some $s_1,s_2 \in \N.$ In other words, we set  $y_j=0$ and $z_k=0$ for $j>s_1$ and $k>s_2$, respectively. We start by defining the truncated weak formulation problem: for every $\bsy^{s_1},\,\bsz^{s_2} \in U,$ find $\vu_{\bf s}(\cdot, \bsy^{s_1},\bsz^{s_2})  \in \vV$, with ${\bf s}=(s_1,s_2)$,   such that
\begin{equation}\label{truncated weak}
  \calB  (\bsy^{s_1},\bsz^{s_2};\vu_{\bf s}(\cdot,\bsy^{s_1},\bsz^{s_2}),\vv) = \ell(\vv) \qquad \forall ~\vv \in \vV.
\end{equation}

Thanks to Theorem \ref{thm: unique solution}, the truncated problem \eqref{truncated weak} has a unique solution. Estimating the truncation  error, which is needed for measuring the  QMC finite element error in \eqref{combine},  is the topic of the next theorem.   For brevity, we let 
\[\mu_c(\bsx,\bsy) = \sum_{j=s_1+1}^\infty y_j \psi_j(\bsx)~~{\rm and}~~ 
 \lambda_c(\bsx,\bsz) = \sum_{j=s_2+1}^\infty z_j \phi_j(\bsx), \quad\bsx\in\Omega,\ \bsy,\bsz\in U.\]
\begin{theorem}\label{Truncating error}
Under Assumption \eqref{ass A2}, for every $\vf \in \vV^*$, $\bsy,\bsz \in U$, and
${\bf s}=(s_1,s_2) \in \N^2$, the solution $\vu_{\bf s}$ of the truncated parametric weak problem
\eqref{truncated weak}
satisfies
\[
\|\vu(\cdot,\bsy,\bsz) - \vu_{\bf s}(\cdot,\bsy^{s_1},\bsz^{s_2})\|_{\vV} 
 \le C\,\widehat C  \|\vf\|_{\vV^*},\]
 where 
\[\widehat C= \sum_{j \ge s_1+1} \|\psi_j\|_{L^\infty(\Omega)}
 +\sum_{j \ge s_2+1} \|\phi_j\|_{L^\infty(\Omega)}.\]
Moreover, if \eqref{ass A5}--\eqref{ass A3} are satisfied, and 
if $\calL:\vV \to \R$ is a bounded  linear functional, (that is,  $|\calL(\vw)|\le \|\calL\|_{\vV^*}\|\vw\|_\vV$ for all $\vw \in \vV$), then for every $\bsy,\,\bsz  \in U,$ we have 
\begin{equation}\label{convergence calL u-us}
|\calL(\vu(\cdot, \bsy,\bsz))-\calL(\vu_{\bf s}(\cdot, \bsy^{s_1},\bsz^{s_2}))| 
\le C\, \Big(s_1^{1-\max(1/p,1/\varrho_1)}+ s_2^{1-\max(1/q,1/\varrho_2)}\Big)\,   \|\vf\|_{\vV^*} \|\calL\|_{\vV^*},    
\end{equation}
for some $0<p\,,q\le 1$ (see  \eqref{ass A1}) and for some $0<\varrho_1\,,\varrho_2<1$ (see \eqref{ass A3})\,. Here, the   (generic) constant $C$  depends on  $\Omega$, $\mu_{\max}$, $\mu_{\min}$, and  $\lambda_{\max}$.
  \end{theorem}
\begin{proof}
From the variational formulations in \eqref{para weak} and \eqref{truncated weak}, we notice that 
\[\calB  (\bsy^{s_1},\bsz^{s_2};\vu_{\bf s}(\cdot,\bsy^{s_1},\bsz^{s_2})-\vu(\cdot,\bsy,\bsz),\vv)=\calB  (\bsy-\bsy^{s_1},\bsz-\bsz^{s_2};\vu(\cdot,\bsy,\bsz),\vv).\]
Following the steps in \eqref{eq: bounded} and using the achieved estimate in \eqref{a priori}, we have 
\begin{align*}
|\calB  (\bsy-\bsy^{s_1},\bsz-\bsz^{s_2};\vu(\cdot,\bsy,\bsz),\vv)|
&\le C\max_{\bsx\in\Omega,\ \bsy,\bsz\in U}(|\lambda_c(\bsx,\bsy)|+|\mu_c(\bsx,\bsz)|)\,\|\vu\|_{\vV}\,\|\vv\|_{\vV}
\\
&\le  C\,\widehat C  \| \vf \|_{\vV^*}\|\vv\|_{\vV}\,.
\end{align*}
On the other hand, by the coercivity property in  \eqref{eq: coer B}, we have 
\begin{multline*}
    \calB  (\bsy^{s_1},\bsz^{s_2};\vu_{\bf s}(\cdot,\bsy^{s_1},\bsz^{s_2})-\vu(\cdot,\bsy,\bsz),\vu_{\bf s}(\cdot,\bsy^{s_1},\bsz^{s_2})-\vu(\cdot,\bsy,\bsz)) \\
    \ge C \mu_{\min} \| \vu_{\bf s}(\cdot,\bsy^{s_1},\bsz^{s_2})-\vu(\cdot,\bsy,\bsz)\|_{\vV}^2\,.
\end{multline*}
Combining  the above equations, then the first desired result follows after simplifying by similar terms. To show \eqref{convergence calL u-us}, we simply use the imposed assumption on $\calL$ and the first achieved estimate to obtain 
\begin{align*}
    \big|\calL\big(\vu(\cdot,\bsy,\bsz) - \vu_{\bf s}(\cdot,\bsy^{s_1},\bsz^{s_2})\big)\big|
&\le
\|\calL\|_{\vV^*}
\|\vu(\cdot,\bsy,\bsz) - \vu_{\bf s}(\cdot,\bsy^{s_1},\bsz^{s_2})\|_{\vV}
\le C\,\widehat C\,   \|\vf\|_{\vV^*} \|\calL\|_{\vV^*}.
\end{align*}
Hence, \eqref{convergence calL u-us} is a direct consequence of the above estimate, the Stechkin inequality
\[
\sum_{j \ge s+1} b_j\le C_\varsigma\,
s^{1-\frac{1}{\varsigma}}\Big(\sum_{j \ge 1} b_j^\varsigma\Big)^{\frac{1}{\varsigma}},\quad{\rm for}~~0< \varsigma < 1,
\]
where $\{b_j\}_{j\ge1}$ is a nonincreasing sequence of positive numbers, and assumptions \eqref{ass A5}--\eqref{ass A3}.  
\end{proof} 
\section{Finite element approximation}\label{Sec: FEM}
This section is devoted to introducing the  Galerkin FEM of degree at most $r$ ($r\ge 1$) for the approximation of the solution to the model problem \eqref{para weak} over $\Omega$, and consequently, to problem \eqref{eq:L1}. Stability and error estimates are  investigated. The achieved results in this section are needed  for measuring the  QMC finite element error in \eqref{combine}.

We introduce a family of regular triangulation (made of simplexes)  $\mathcal{T}_h$ of the domain $\overline{\Omega}$ and set $h=\max_{K\in \mathcal{T}_h}(h_K)$, where $h_{K}$ denotes the diameter of the element $K$. Let $V_h \subset H^1_0(\Omega)$ denote the usual conforming finite element space of continuous, piecewise polynomial functions of degree at most $r$ on $\mathcal{T}_h$ that vanish on $\partial \Omega$. Let $\vV_h=[V_h]^d$ be the finite element vector space. Then there exists a
constant $C$ (depending on $\Omega$)   such that,
\begin{equation}\label{projection estimate}
    \inf_{\vv_h \in \vV_h} \| \vv -\vv_h\|_{\vV} \le Ch^\theta \|\vv\|_{{\bf H}^{\theta+1}},\quad {\rm for}~~1\le \theta \le r.
\end{equation}

Motivated by the weak formulation in \eqref{para weak},  we define the parametric finite element approximate solution  as:  find $\vu_h(\cdot, \bsy,\bsz) \in \vV_h$ such that 
\begin{equation}\label{FE solution}
 \calB(\bsy,\bsz;\vu_h, \vv_h) = \ell(\vv_h), \quad \text{for all} \quad \vv_h \in \vV_h,\quad{\rm for~every}~~\bsy,\,\bsz \in U.
\end{equation}

Assuming that \eqref{ass A2} is satisfied, then, for every $\vf \in \vV^*$ and every $\bsy,\,\bsz  \in U$, the finite element scheme defined in \eqref{FE solution} has a unique parametric solution $\vu_h(\cdot, \bsy,\bsz) \in \vV_h$.  This can be shown by mimicking the proof of Theorem \ref{thm: unique solution} because $\vV_h \subset \vV.$ Furthermore, the finite element solution  is also stable; the bound in \eqref{a priori} remains valid with $\vu_h$ in place of $\vu$, that is,  
\begin{equation}\label{eq: H1 bound of u_h}
\| \vu_h(\cdot, \bsy,\bsz)\|_{\vV} \le C\mu_{\min}^{-1} \|\vf\|_{\vV^*}\,.
\end{equation}  

In the next theorem, we  discuss the $\vV$-norm error estimate from the finite element discretization. Then, and as in Theorem \ref{Truncating error},  for measuring the  QMC finite element error in \eqref{combine}, we derive an estimate that involves a linear functional $\cal$ acting on the difference $\vu-\vu_h.$   
\begin{theorem}\label{Convergence theorem}
For every $\bsy,\,\bsz  \in U,$ let $\vu$ and $\vu_h$ be the solutions of problems \eqref{eq:L1} and \eqref{FE solution}, respectively.  Assuming that $\vu$ satisfies the regularity properties in \eqref{a priori H2} for some integer $1\le \theta\le r.$ Under Assumption \eqref{ass A2} and \eqref{ass A4}, and when $\vf \in {\bf H}^{\theta-1}(\Omega)$, we have
\begin{equation}\label{convergence}
    \|\vu(\cdot, \bsy,\bsz)-\vu_h(\cdot, \bsy,\bsz)\|_{\vV}  
\le  C\, h^\theta \|\vf\|_{{\bf H}^{\theta-1}}.
\end{equation}
Moreover, if $\calL:{\bf L}^2(\Omega) \to \R$ is a bounded  linear functional (that is, $|\calL(\vw)|\le \|\calL\|\,\|\vw\|$), then 
\begin{equation}\label{convergence calL}
|\calL(\vu(\cdot, \bsy,\bsz))-\calL(\vu_h(\cdot, \bsy,\bsz))|\le  C\, h^{\theta+1} \|\vf\|_{{\bf H}^{\theta-1}} \|\calL\|,\quad{\rm for~every}~~\bsy,\,\bsz  \in U.
\end{equation}
 The  constant $C$  depends on  $\Omega$, $\mu_{\max}$, $\mu_{\min}$, and  $\lambda_{\max}$, but not on $h$.
\end{theorem}
\begin{proof} The proof of \eqref{convergence} follows a standard argument for finite element approximations and is included here for completeness.
From the weak formulation in \eqref{para weak} and the finite element scheme  in \eqref{FE solution}, we have the following orthogonality property
\begin{equation}\label{equ:Gal ort}
\calB(\bsy,\bsz;\vu-\vu_h, \vv_h) = 0, \quad \text{for all} \quad \vv_h \in \vV_h\,.
\end{equation}
By using the above equation, the coercivity and boundedness of $\calB(\bsy,\bsz;\cdot, \cdot)$, we obtain, 
\begin{align*}
    \| \vu-\vu_h \|_{\vV}^2 
&\le 
C \calB(\bsy,\bsz;\vu-\vu_h, \vu-\vu_h)
=
C \calB(\bsy,\bsz;\vu-\vu_h, \vu-\vw_h)
\le
C \|\vu-\vu_h\|_{\vV}\,\|\vu-\vw_h\|_{\vV},
\end{align*}
for all $\vw_h\in\vV_h$. This implies $\| \vu-\vu_h \|_{\vV}
\le C \|\vu-\vw_h\|_{\vV}$ for all $\vw_h\in\vV_h$. Thus, \eqref{convergence} follows from \eqref{projection estimate} and the regularity estimate in \eqref{a priori H2}.

To show \eqref{convergence calL}, we use the so-called Nitsche trick. We first replace $\ell$ in \eqref{para weak} by $\calL$, and consider a new parametric variational problem: find $\vu_{\calL} \in \vV$ such that  
\begin{equation}\label{para weak new}
 \calB(\bsy,\bsz;\vu_{\calL}, \vv) = \calL(\vv), \quad \text{for all} \quad \vv \in \vV.
\end{equation}
By Theorem \ref{thm: unique solution}, this problem 
has a unique solution for every $\bsy,\bsz \in U$. Hence, using Theorem~\ref{lem: vu bound} (with $\vu_{\calL}$ in place of $\vu$) and the given assumption on  $\calL$, we conclude that  $\|\vu_{\calL}\|_{{\bf H}^2}\le C\|\calL\|.$ Therefore, by repeating the above argument, we deduce
\begin{equation}\label{equ:uL}
\|\vu_{\calL}(\cdot,\bsy,\bsz)-\vu_{\calL,h}(\cdot,\bsy,\bsz)\|_{\vV}\le Ch\|\vu_{\calL}\|_{{\bf H}^2}\le Ch\|\calL\|,
\end{equation}
where $\vu_{\calL,h} \in \vV_h$ is the finite element approximation of $\vu_{\calL}$. By using successively the linearity of $\calL$, equation~\eqref{para weak new}, the symmetry of $\calB(\bsy,\bsz;\cdot,\cdot)$, the Galerkin orthogonality \eqref{equ:Gal ort}, and the boundedness of $\calB(\bsy,\bsz;\cdot,\cdot)$, we obtain
\begin{align*}
|\calL(\vu(\cdot,\bsy,&\bsz))-\calL(\vu_h(\cdot,\bsy,\bsz))|
=
|\calL(\vu(\cdot,\bsy,\bsz)-\vu_h(\cdot,\bsy,\bsz))|
\\
&=
|\calB(\bsy,\bsz;\vu-\vu_h,\vu_{\calL})|
=
|\calB(\bsy,\bsz;\vu-\vu_h,\vu_{\calL}-\vu_{\calL,h})|
\le
C \|\vu-\vu_h\|_{\vV}\,\|\vu_{\calL}-\vu_{\calL,h}\|_{\vV}.
\end{align*}
The required estimate \eqref{convergence calL} now follows from \eqref{convergence} and \eqref{equ:uL}\,.
\end{proof}

\section{QMC method and sparse grids}\label{sec: QMC errors}
Our aim is to measure the QMC finite element error which occurs in the  third term on the right hand side of \eqref{combine}. To serve this purpose, the current  section is dedicated to investigate the high-order QMC   and the high-order QMC sparse grid errors from estimating the finite dimensional   integral
\begin{equation}\label{If}
{\mathcal I}_{\bf s} {F}:=\int_{U_2} \int_{U_1} {F}(\bsy,\bsz)\,d\bsy\,d\bsz\,.\end{equation}
 Recall that   $U_i=[0,1]^{s_i}$ are of fixed dimensions $s_i$ for $i=1,2,$ and ${\bf s}=(s_1,s_2)$. We approximate ${\mathcal I}_{\bf s} {F}$ via an 
equal-weight rule of the form:   
\begin{equation}\label{eq:QMCInt}
Q_{\bf s,\bf N} [{F}]:=\frac{1}{N_1\,N_2}\sum_{k=0}^{N_2-1}\sum_{j=0}^{N_1-1} {F}(\bsy_j,\bsz_k),\quad{\rm with}~~{\bf N}=(N_1,N_2),
\end{equation}
where  $N_i=b^{m_i}$ for a given prime $b$ and a given positive integer $m_i$, with $i=1,2.$ The QMC points  $\{\bsy_0, \ldots,\bsy_{N_1-1} \}$ belong to $ U_1$ and $\{\bsz_0, \ldots,\bsz_{N_2-1} \}$ belong to $U_2$. We shall analyze, in particular, $Q_{\bf s,\bf N}$ being
deterministic, interlaced high-order polynomial lattice rules as introduced in \cite{Dick2008} and as considered for affine-parametric operator equations in \cite{DickKuoGiaNuynsSchwab2014}. To this end, to generate a polynomial lattice rule in base $b$ with $N_1$ points in $U_1$, we need a \emph{generating vector} of polynomials $\bg(x) = (g_1(x), \ldots, g_{s_1}(x)) \in [P_{m_1}({\mathbb Z}_{b})]^{s_1}$, where  $P_{m_1}({\mathbb Z}_{b})$ is the space of polynomials of degree less than $m_1$ in $x$ with coefficients taken from a finite field ${\mathbb Z}_{b}$.

For each integer $0\le n\le b^{m_1}-1$, we associate $n$ with the polynomial
\[ n(x) = \sum_{i=1}^{m_1} \eta_{i-1} x^{i-1}  \quad \in {\mathbb Z}_{b}[x],\]
where $(\eta_{m_1-1}, \ldots ,\eta_0)$ is the $b$-adic expansion of $n$, that is $n =\sum_{i=1}^{m_1} \eta_{i-1}\,b^{i-1}\,.$ We also need a map $v_{m_1}$ which maps elements in ${\mathbb Z}_{b}(x^{-1})$ to
the interval $[0,1)$, defined for any integer $w$ by
\[v_{m_1} \left( \sum_{\ell=w}^\infty t_{\ell} x^{-\ell} \right) =\sum_{\ell=\max(1,w)}^{m_1} t_{\ell} b^{-\ell}.\]

Let $P \in {\mathbb Z}_b[x]$ be an irreducible polynomial with degree $m_1$. The classical polynomial lattice rule $\calS_{P,b,m_1,s_1}(\bg)$ associated with $P$ and the generating vector $\bg$ is comprised of the quadrature points
\[\bsy_n =\left(v_{m_1} \Big( \frac{n(x)g_j(x)}{P(x)} \Big)\right)_{1\le j\le s_1} \in [0,1)^{s_1},\quad {\rm for}~~n = 0,\ldots, N_1 - 1.\]
In a similar fashion, we define the quadrature points $\bsz_n \in [0,1)^{s_2}$ for $n=0,\ldots,N_2-1$. In this case, the \emph{generating vector} of polynomials is $\bg = (g_1, \ldots, g_{s_2})$.

Classical polynomial lattice rules give almost first order of convergence for integrands of bounded variation. To obtain high-order of convergence, an interlacing procedure described as follows is needed. Following \cite{Goda2015,GodaDick2015}, the \emph{digit interlacing function}  with digit interlacing factor $\alpha \in \mathbb{N}$, $\mathscr{D}_\alpha: [0,1)^{\alpha}  \to  [0,1)$,   is defined  by
\[\mathscr{D}_\alpha(x_1,\ldots, x_{\alpha})= \sum_{i=1}^\infty \sum_{j=1}^\alpha
\xi_{j,i} b^{-j - (i-1) \alpha}\;,\]
where $x_j = \sum_{i\ge 1} \xi_{j,i}\, b^{-i}$ for $1 \le j \le \alpha$.
For vectors, we set $\mathscr{D}^s_\alpha: [0,1)^{\alpha s}  \to  [0,1)^s$ with 
\[\mathscr{D}^s_\alpha(x_1,\ldots, x_{\alpha s}) =
(\mathscr{D}_\alpha(x_1,\ldots, x_\alpha),  \ldots,
\mathscr{D}_\alpha(x_{(s-1)\alpha +1},\ldots, x_{s \alpha}))\;.\]
Then, an interlaced polynomial lattice rule of
order $\alpha$ with $b^m$ points in $s$ dimensions
is a QMC rule using $\mathscr{D}_\alpha(\calS_{P,b,m,\alpha s}(\bg))$
as quadrature points, for some given modulus $P$ and generating vector $\bg$.

Next, we derive the error from approximating the  integral ${\mathcal I}_{\bf s} {F}$ in \eqref{If} by the QMC quadrature formula $Q_{\bf s,\bf N} [{F}]$ in \eqref{eq:QMCInt}. The proof  relies on \cite[Theorem 3.1]{DickKuoGiaNuynsSchwab2014}.
\begin{theorem}\label{prop:qmc}
 Let $\bchi=(\chi_j)_{j\ge 1}$ and $\bvphi = (\varphi_j)_{j\ge 1}$ be two sequences of positive numbers with $\sum_{j=1}^\infty \chi_j^p$ and $\sum_{j=1}^\infty \varphi_j^q$ being finite for some $0<p,\,q<1$. Let $ \bchi_{s_1}=(\chi_j)_{1\le j\le s_1}$ and $\bvphi_{s_1} = (\varphi_j)_{1\le j\le s_2},$ and let $\alpha \,:=\, \lfloor 1/p \rfloor +1$ and $\beta \,:=\, \lfloor 1/q \rfloor +1$. Assume that $F$ satisfies the following regularity properties: for any $\bsy\in U_1$, $\bsz\in U_2$, 
$\balpha \in \{0, 1, \ldots, \alpha\}^{s_1}$, and $\bbeta \in \{0, 1, \ldots, \beta\}^{s_2}$, the following inequalities hold 
\begin{equation} \label{eq:like-norm}
|\partial^{\balpha}_\bsy {F}(\bsy,\bsz)| \le c|\balpha|! 
\bchi_{s_1}^{\balpha}\quad{\rm and}\quad |\partial^{\bbeta}_\bsz {F}(\bsy,\bsz)| \le c|\bbeta|! 
\bvphi_{s_2}^{\bbeta}, 
\end{equation}
where the constant $c$ is independent of $\bsy$, $\bsz$, $s_1$, $s_2$, and of $p$ and $q.$ Then one can construct two interlaced polynomial lattice rules of order $\alpha$ with $N_1$ points, and of order $\beta$ with $N_2$ points, using a fast component-by-component (CBC) algorithm,  with cost $\calO(\alpha\,s_1 N_1(\log N_1+\alpha\,s_1))$ and $\calO(\beta\,s_2 N_2(\log N_2+\beta\,s_2))$  operations, respectively,  so that the following error bound holds
\[|{\mathcal I}_{\bf s} {F}-Q_{\bf s,\bf N} [{F}]|\le C\,\Big(N_1^{-1/p}+N_2^{-1/q}\Big)\,.\]
The  constant $C$ depends on $b,p$ and $q$, but is independent of $s_i$ and $m_i$ for $i\in \{1,2\}.$ 

By choosing $m_1, m_2 \in \mathbb{N}$ such that $|m_1q - m_2 p| < 1$, we obtain that the total number of QMC points is $N = N_1 N_2 = b^{m_1+m_2}$ and 
\[|{\mathcal I}_{\bf s} {F}-Q_{\bf s,\bf N} [{F}]|\le C\, N^{-\frac{1}{p+q}}\,.\] 
\end{theorem}
\begin{proof}
By adding and subtracting $Q_{s_1,N_1}[{F}(\cdot,\bsz)]:=\frac{1}{N_1}\sum_{j=0}^{N_1-1} {F}(\bsy_j,\bsz),$ the QMC error can be decomposed as
\begin{multline*}
|{\mathcal I}_{\bf s} {F}-Q_{\bf s,\bf N} [{F}]|\le \int_{U_2} \Big|\int_{U_1} {F}(\bsy,\bsz)\,d\bsy -Q_{s_1,N_1}[{F}(\cdot,\bsz)]\Big|\,d\bsz\\
+\frac{1}{N_1}\sum_{j=0}^{N_1-1}\Big|\int_{U_2} {F}(\bsy_j,\bsz)\,d\bsz-Q_{s_2,N_2}[{F}(\bsy_j,\cdot)]
\Big|\,,
\end{multline*}
where $Q_{s_2,N_2}[{F}(\bsy_j,\cdot)]:=\frac{1}{N_2}\sum_{k=0}^{N_2-1}
 {F}(\bsy_j,\bsz_k)$.  By using \cite[Theorem 3.1]{DickKuoGiaNuynsSchwab2014} and the regularity assumptions in \eqref{eq:like-norm}, we have 
\[\Big|\int_{U_1} {F}(\bsy,\bsz)\,d\bsy -Q_{s_1,N_1}[{F}(\cdot,\bsz)]\Big| \le C\, N_1^{-1/p}~{\rm and}~\Big|\int_{U_2} {F}(\bsy_j,\bsz)\,d\bsz-Q_{s_2,N_2}[{F}(\bsy_j,\cdot)]\Big|\le C N_2^{-1/q}\,.\]
Combining the above equations,  we immediately deduce the first desired results.

Now, using  $N= b^{m_1+m_2}$ and the conditions  $|m_1 q - m_2 p| < 1,$ we have 
\begin{multline*}
N_1^{-1/p} = N^{-1/(p+q)} b^{(m_1+m_2)/(p+q)-m_1/p}\\
= N^{-1/(p+q)} b^{(m_2p-m_1q)/(p(p+q))}\le N^{-1/(p+q)} b^{1/(p(p+q))}\le C N^{-1/(p+q)}\,.    \end{multline*}
Similarly, 
\begin{equation*}
N_2^{-1/q}\le N^{-1/(p+q)} b^{1/(q(p+q))}\le CN^{-1/(p+q)}\,,
\end{equation*}
and therefore,  the proof of the second desired estimate is completed. 
\end{proof}

In order to reduce the computational cost (and thereby improving the convergence rate), we next discuss a combination of the QMC rules with a sparse grid approach. The sparse grid approach gives us more flexibility as we can combine different polynomial lattice rules. Since the weights for the expansion of $\lambda$ and $\mu$ are of a similar form as in other problems \cite{DickKuoGiaNuynsSchwab2014}, we can also reuse existing constructions of higher order polynomial lattice rules.

Let $\{N_i^{(j)}\}_{j\ge 1}$  be increasing sequences of positive values for $i=1,2.$ Then
\begin{equation*}
\mathcal{I}_{\bf s} {F} = \lim_{j,k \to \infty} Q_{\mathbf{s}, {\bf N}^{j,k}}[{F}],\quad {\rm with}~~~{\bf N}^{j,k}=(N_1^{(j)}, N_2^{(k)}).
\end{equation*}
We can write this as a telescoping sum
\begin{equation}\label{eq: infinite sums}
\mathcal{I}_{\bf s} {F} = \sum_{j, k = 1}^\infty a_{jk},~{\rm with}~a_{jk}= \overline a_{j,k}-\overline a_{j,k-1}=  \widehat a_{j,k}-\widehat a_{j-1,k},
\end{equation}
where 
\[\overline a_{j,k}=Q_{\mathbf{s},{\bf N}^{j,k}}[F] - Q_{\mathbf{s}, {\bf N}^{j-1,k}}[F],\quad \widehat a_{jk}= Q_{\mathbf{s},{\bf N}^{j,k}}[F] - Q_{\mathbf{s},{\bf N}^{j,k-1}}[F], \]
and 
\begin{equation}\label{eq: cond}
    Q_{\mathbf{s}, {\bf N}^{j,k}} = 0~~{\rm if}~~ (j,k) \in \{(\zeta,0), (0,\omega): \zeta, \omega \in \N\cup \{0\}\}\,.
\end{equation}

To get a computable quantity, we need to truncate the infinite sums in \eqref{eq: infinite sums}. This can be done in different ways; we choose to truncate the tensor grid along the main diagonal of indexed points for each combination of the QMC rules ${\bf N}^{j,k}$ where $L$ is the ``level'' of the sparse grid rule. Explicitly, we truncate as follows: 
\begin{equation}\label{eq: truncate infinite sums}
{\mathcal I}_{{\bf s},L}[{F}] = \sum_{\substack{j, k = 1 \\ j+k \le L }} a_{jk}
  =\sum_{k = 1}^{L-1} \left( Q_{\mathbf{s},{\bf N}^{L-k,k}}[{F}]  - Q_{\mathbf{s},{\bf N}^{L-k,k-1}}[{F}]  \right),\quad   
\end{equation}
where in the second equality we used  $ \sum_{\substack{j, k = 1 \\ j+k \le L }}=\sum_{k = 1}^{L-1}\sum_{j = 1}^{L-k}$ and the condition in \eqref{eq: cond}.  We prove next that the quadrature error incurred on this QMC sparse  grid is relatively small for a sufficiently large $L$. 

In the next theorem, for some $\vartheta > 0$ such that $p \vartheta, q \vartheta \ge 1$, we choose  $N_1^{(j)}=b^{ \lceil j p \vartheta \rceil}$ and $N_2^{(j)}=b^{ \lceil j q \vartheta \rceil}$ for $j\ge 1$. The purpose of $\vartheta > 0$ is to avoid a situation where $N_i^{(j)} = N_i^{(j+1)}$ for some admissible $i, j$. Choosing $\vartheta$ such that $p \vartheta, q \vartheta \ge 1$ guarantees that this cannot happen. On the other hand, since the constant $C$ increases with $\vartheta$, we consider $\vartheta$ as a constant. In other words, in order to reduce the error in Theorem~\ref{qmc sparse grid} one increases $L$ and therefore $M$, but keeps $\vartheta$ fixed. 
\begin{theorem}\label{qmc sparse grid}
 Under the assumptions of Theorem \ref{prop:qmc}, the following error bound holds true   $|{\mathcal I}_{\bf s}[{F}]-{\mathcal I}_{{\bf s},L}[{F}]|
\le  C b^{-L\vartheta/2},$ where the constant $C$ depends on $b,p,q, \vartheta$, but is independent of $s_1$, $s_2$ and $L$.

Furthermore, we let $M$ denote the total number of quadature points used in the QMC sparse grid rule, and  then
\begin{equation*}
|{\mathcal I}_{\bf s}[{F}]-{\mathcal I}_{{\bf s},L}[{F}]|
\le C (\log M)^{1_{p=q}/(2p)} M^{-\min(1/p, 1/q)/2},
\end{equation*}
where $1_{p=q}$ is $1$ if $p=q$ and $0$ otherwise, and where the constant $C$ depends on $b, p, q, \vartheta$, but is independent of $s_1$, $s_2$ and $M$.
\end{theorem}
\begin{proof}
Let $\omega>L$ be any positive integer.   Decomposing as  
\begin{equation}\label{3terms}
    \sum_{\substack{j, k = 1 \\ j+k \ge L+1 }}^\omega a_{jk}=\sum_{k = L}^\omega\sum_{j=1}^\omega  a_{jk}+\sum_{k = 1}^{\lceil L/2\rceil-1}\, \sum_{j=L-k+1}^{\omega} a_{jk}+\sum_{k = \lceil L/2\rceil}^{L-1}\,\,\sum_{j=L-k+1}^\omega a_{jk}\,,
\end{equation}
and the task now is to estimate the three terms on the right-hand side of this equality. Using the definition of $a_{jk}$ in \eqref{eq: infinite sums}, we notice that 
\[\sum_{j=1}^\omega a_{jk}   =\widehat a_{\omega k}=Q_{s_2,N_2^{(k)}}[{\bf g}_\omega] - Q_{s_2,N_2^{(k-1)}}[{\bf g}_\omega],~~{\rm with}~~{\bf g}_\omega(\bsz)= Q_{s_1, N_1^{(\omega)}}[{F}(\cdot,\bsz)]\,.\]
However, by adding and subtracting  $\int_{U_2}{\bf g}_{m_0}(\bsz)\,d\bsz$ and using  \cite[Theorem 3.1]{DickKuoGiaNuynsSchwab2014} (where the regularity assumption in \eqref{eq:like-norm} is needed here), we obtain 
\begin{equation}\label{differenceQ}
   |\widehat a_{m_0 k}|\le \sum_{\ell=k-1}^k\Big|Q_{s_2, N_2^{(\ell)}}[\mathbf{g}_{m_0}]-\int_{U_2}{\bf g}_{m_0}(\bsz)\,d\bsz\Big|\le C \Big(N_2^{(k-1)}\Big)^{-1/q}\,,\end{equation}
   for any positive integer $m_0,$ and consequently, 
 \begin{equation}\label{term1}
\Big|\sum_{k= L}^\omega\sum_{j=1}^\omega a_{jk} \Big|    \le 
\sum_{k= L}^\omega|\widehat a_{\omega k}| \le C \sum_{k= L}^\omega\Big(N_2^{(k-1)}\Big)^{-1/q}\le C \sum_{k= L}^\infty\Big(N_2^{(k-1)}\Big)^{-1/q}\,.\end{equation}
By using \eqref{differenceQ} twice (for $m_0=\omega$ and $m_0=L-k$),   we obtain the following estimate of the third term in \eqref{3terms};  
\begin{equation}\label{term3}
\sum_{k = \lceil L/2\rceil}^{L-1}\,\Big| \sum_{j=L-k+1}^{\omega} a_{jk}\Big|
=\sum_{k = \lceil L/2\rceil}^{L-1}\,|\widehat a_{\omega k}-\widehat a_{L-k,k}|\le C\sum_{k = \lceil L/2\rceil}^{L-1} \Big(N_2^{(k-1)}\Big)^{-1/q}\,.
\end{equation}
The remaining task is to estimate the middle term on the right-hand side of \eqref{3terms}. By shifting and changing the order of summations, we get 
\begin{align*}
    \sum_{k = 1}^{\lceil L/2\rceil-1}\,\sum_{j=L-k+1}^{\omega} a_{jk}&=\sum_{k = 1}^{\lceil L/2\rceil-1}\,\sum_{j=L-k+1}^{\omega} [\overline a_{jk}-\overline a_{j,k-1}]
=\sum_{j=L-\lceil L/2\rceil+2}^{\omega} \overline a_{j,\lceil L/2\rceil-1}
-\sum_{k = 1}^{\lceil L/2\rceil-2}\overline a_{L-k,k}\,.
\end{align*}
Now,  proceeding as in \eqref{differenceQ}  but on the region $U_1$, where the regularity assumption in \eqref{eq:like-norm} is needed here, we have 
\[ |\overline a_{j,m_0}|\le \sum_{\ell=j-1}^j\Big|Q_{s_1, N_1^{(\ell)}}[\mathbf{g}_{m_0}]-\int_{U_1}{\bf g}_{m_0}(\bsz)\,d\bsz\Big|\le C \Big(N_1^{(j-1)}\Big)^{-1/p}\,,\]
   for any positive integer $m_0,$ and with ${\bf g}_{m_0}(\bsy)= Q_{s_2, N_2^{(m_0)}}[{F}(\bsy,\cdot)]$. Therefore,  
\begin{multline*}
    \Big|\sum_{k = 1}^{\lceil L/2\rceil-1}\, \sum_{j=L-k+1}^{\omega} a_{jk}\Big| \\
    \le 
C\sum_{j=\lceil L/2\rceil}^\infty \Big(N_1^{(j-1)}\Big)^{-1/p}
+C\sum_{k = 1}^{\lceil L/2\rceil-2}\Big(N_1^{(L-k-1)}\Big)^{-1/p}
\le 
C\sum_{j=\lceil L/2\rceil}^\infty \Big(N_1^{(j-1)}\Big)^{-1/p}\,.
\end{multline*}
Inserting this, \eqref{term1} and  \eqref{term3}  in \eqref{3terms}  leads to 
\[\Big|\sum_{\substack{j, k = 1 \\ j+k \ge L+1 }}^\omega a_{jk}\Big|\le 
C\sum_{j=\lceil L/2\rceil}^\infty \Big(N_1^{(j-1)}\Big)^{-1/p}+C\sum_{j=\lceil L/2\rceil}^\infty \Big(N_2^{(j-1)}\Big)^{-1/q},\]
for any positive integer $\omega > L.$ Thus, using  $N_1^{(j)}=b^{ \lceil j p \vartheta \rceil}$ and $N_2^{(j)}=b^{ \lceil jq \vartheta \rceil}$, we get 
\begin{multline*}
    |{\mathcal I}_{\bf s}[{F}]-{\mathcal I}_{{\bf s},L}[{F}]|
 =\Big|\sum_{\substack{j, k = 1 \\ j+k \ge L+1 }}^\infty a_{jk}\Big|\\
 \le 
C\sum_{j=\lceil L/2\rceil}^\infty \Big[\Big(N_1^{(j-1)}\Big)^{-1/p}+\Big(N_2^{(j-1)}\Big)^{-1/q}\Big]
\le  C\,\sum_{j=\lceil L/2\rceil}^\infty b^{-j \vartheta}
\le  C b^{-L\vartheta/2}\,.
\end{multline*}
Hence the first desired QMC sparse grid estimate is obtained. 

The total number of quadrature points used in the QMC sparse grid approach is
\begin{equation*}
M = \sum_{k=1}^{L-1} b^{\lceil (L-k) p \vartheta \rceil} b^{ \lceil k q \vartheta \rceil} \le b^{2+Lp\vartheta} \sum_{k=1}^{L-1} b^{k(q-p) \vartheta }.
\end{equation*}
For $q = p$ we have $M\le b^{2+Lp\vartheta }(L-1),$ and for $q \ne p$ we have
\begin{align*}
M \le b^{2+Lp \vartheta} \frac{b^{(q-p) \vartheta }-b^{L(q-p)\vartheta}}{1-b^{ (q-p) \vartheta }} \le    
      \frac{ b^{2+L\max(p,q) \vartheta -|p-q| \vartheta }}{1-b^{-|p-q| \vartheta }} \le C\,b^{L\max(p,q) \vartheta}\,.
\end{align*}
Since the error is of order $b^{-L \vartheta/2 }$ we have
\begin{equation*}
|{\mathcal I}_{\bf s}[{F}]-{\mathcal I}_{{\bf s},L}[{F}]|^2
\le C b^{-L \vartheta } \leq C L^{1_{p=q}/p} M^{-\min(1/p, 1/q)}.
\end{equation*}
 Since $M \ge b^{(L-1) p \vartheta }$, $L \le C \log M$, and hence, the second desired  result follows.
\end{proof}

In the next theorem, we state the error from approximating the  integral ${\mathcal I}_{\bf s} {F}$ in \eqref{If} by a QMC rule in dimension $s_1+s_2$ directly, without combining QMC with sparse grids. In this approach we combine the different weights arising from simulating $\mu$ and $\lambda$.
 The proof follows directly from \cite[Theorem 3.1]{DickKuoGiaNuynsSchwab2014}.
\begin{theorem}\label{prop:qmc6.3}
 Let $\gamma=\min(\lfloor 1/p \rfloor,\lfloor 1/q \rfloor) +1$, and let $\bchi$ and $\bvphi$ be the two sequences introduced in Theorem \ref{prop:qmc}.  For any $\bsr=(\bsy,\bsz)\in U_1\times U_2=[0,1]^{s_1+s_2}$ and any  $\boldsymbol{\gamma} \in \{0, 1, \ldots, \gamma\}^{s_1+s_2}$, assume that $F$ satisfies
 \begin{equation} \label{eq:like-norm mixed}
|\partial^{\boldsymbol{\gamma}}_{\bsr} {F}(\bsr)|=|\partial^{\boldsymbol{\gamma}_1,\boldsymbol{\gamma}_2}_{\bsy,\bsz} {F}(\bsy,\bsz)| \le c|\boldsymbol{\gamma}|! 
\bchi_{s_1}^{\boldsymbol{\gamma}_1} 
\bvphi_{s_2}^{\boldsymbol{\gamma}_2},  
\end{equation}
where the vectors $\boldsymbol{\gamma}_1$ and $\boldsymbol{\gamma}_2$ formed  of the first $s_1$ and last $s_2$ components of $\boldsymbol{\gamma},$ respectively, and  the constant $c$ is independent of $\bsr$, $s_1$, $s_2$, and of $p$ and $q.$ Then one can construct an interlaced polynomial lattice rules of order $\gamma$ with $N=b^m$ (for a given prime $b$ and a given positive integer $m$) points using a fast CBC algorithm,  with cost $\calO(\gamma\,(s_1+s_2) N(\log N+\gamma\,(s_1+s_2)))$  operations,  so that 
\[|{\mathcal I}_{\bf s} {F}-Q_{{\bf s},N} [{F}]|=\Big|{\mathcal I}_{\bf s} {F}-\frac{1}{N} \sum_{n=0}^{N-1} F(\boldsymbol{r}_n)\Big|\le C\,N^{-\min(1/p,1/q)}\,,\]
where the generic constant $C$ depends on $b,p$ and $q$, but is independent of ${\bf s}$ and $m.$ 
\end{theorem}

Although the theorem does not require the components of the interlaced polynomial lattice rule to be ordered in a certain way, in practice it is beneficial to order the components such that the early components of the polynomial lattice rule are applied to the most important coefficients in the expansions of $\lambda$ and $\mu$.

\section{Numerical experiments}\label{Sec: Numeric}
In this section, we illustrate numerically the theoretical finding in Theorem  \ref{main results}.  In all experiments,  $\Omega$ is chosen to be the unit square $[0,1]^2$, and  $\mathcal{T}_h$ is a family of uniform triangular meshes with diameter $\sqrt{2} h$ obtained from uniform $J$-by-$J$ square meshes by cutting each mesh square into two congruent triangles with $h=1/J.$  In all numerical experiments we set the base of the polynomial lattice rules $b = 2$.


{\bf Example 1:} In this example, we corroborate the finite element errors and convergence rates when $r=1$ (piecewise linear Galerkin FEM) the   Lam\'e parameters $\mu$ and $\lambda$ are  variable but deterministic.  We choose  \[\mu(x_1,x_2)=x_1+x_2+1\quad{\rm and}\quad  \lambda(x_1,x_2)=\sin(2\pi x_1)+2.\]

To illustrate the FEM convergence order in Theorem \ref{Convergence theorem} (or Theorem \ref{main results}), we choose the body force $\vf$ so that the exact solution is
\[\vu(x_1,x_2)=\left[\begin{matrix} 
u_1(x_1,x_2)\\ 
u_2(x_1,x_2)\end{matrix}\right]=\left[\begin{matrix} 
2(\cos(2\pi x_1)-1)\sin(2\pi x_2)\\ 
(1-\cos(2\pi x_2)) \sin(2\pi x_1)\end{matrix}\right]\,.\] 
Motivated by the equality 
\begin{equation}\label{Ritz}
\|\vv\|= \sup_{\vw\in {\bf L}^2(\Omega),\vw\ne {\bf 0}} \frac{|\iprod{\vv,\vw}|}{\|\vw\|},\quad{\rm for ~any}~~\vv \in {\bf L}^2(\Omega),
\end{equation}
we define, for some fixed (but arbitrary) $\vw\in {\bf L}^2(\Omega)$, the functional $\calL$ by: $\calL(\vv)= \calL_\vw(\vv):=\iprod{\vw,\vv}$. Then, by using the convergence estimate~\eqref{convergence calL} in Theorem \ref{Convergence theorem}, we have $|\iprod{\vu-\vu_h,\vw}|\le  C\, h^2 \|\vf\| \|\vw\|$ for $\vw\in {\bf L}^2(\Omega)\,.$ Consequently, the equality in \eqref{Ritz}  leads to the following optimal   ${\bf L}^2(\Omega)$ estimate: $E_h:=\|\vu-\vu_h\|\le C\,h^2 \|\vf\|\,.$ To demonstrate this numerically, we compute 
$E_{h}$ by approximating the  ${\bf L}^2$-norm ($\|\cdot\|$) using the centroids of the elements in the mesh $\mathcal{T}_h$. The empirical convergence rate (CR) is calculated by halving $h$, and thus, $\text{CR}=\log_2(E_h/E_{h/2}).$ 

If we choose $\vw={\bf 1}$ (the unitary constant vector function), then with $\vv=[v_1~~ v_2]^T,$
\[\calL(\vv)=\calL_{\bf 1}(\vv)=\int_{\Omega} \vv(\bsx)\cdot {\bf 1} \,d\bsx=\int_{\Omega} [v_1(\bsx)+v_2(\bsx)]\,d\bsx\,,\]
which is the mean of $\vv$ over $\Omega=[0,1]^2.$ Since $\|\calL_{\bf 1}\|\le 1,$  by Theorem \ref{Convergence theorem}, 
\[
|\calL(\vu-\vu_h)|=|\calL_{\bf 1}(\vu-\vu_h)|
= \Big|\sum_{i=1}^2\int_\Omega (u_i-u_{i_h})\,d\bsx\Big|\le  C\, h^2 \|\vf\|\,.\]

Again, we use the centroids of the elements in the mesh $\mathcal{T}_h$ to approximate the above integral. The reported  numerical (empirical) convergence rates in Table \ref{Tab 1} illustrate the expected second order of accuracy. For a graphical illustration of the efficiency of the approximate solution over the global domain $\Omega$, we highlight the pointwise nodal displacement errors in Figure \ref{Fig 2} for $J=60.$

\begin{table}[ht]
\begin{center}
\begin{tabular}{|c|cc|cc|cc|}
\hline
$J$&\multicolumn{2}{c|}{$\|\vu-\vu_h\|$~~~~~CR }
&\multicolumn{2}{c|}{$|\calL_{\bf 1}(\vu-\vu_h)|$~~~~~CR}\\
\hline
  8&   3.8533e-01 &          &   1.1697e-02 &         \\
 16&   1.1163e-01 &   1.7873 &   3.7017e-03 &  1.6599 \\
 32&   2.9204e-02 &   1.9345 &   9.8934e-04 &  1.9037 \\
 64&   7.3903e-03 &   1.9825 &   2.5179e-04 &  1.9743 \\
128&   1.8533e-03 &   1.9955 &   6.3238e-05 &  1.9933 \\  
      \hline
\end{tabular}     
\caption{Example 1, Errors  and empirical convergence rates for different values of $J.$}
\label{Tab 1}
\end{center}
\end{table}

\begin{figure}[ht]
\begin{center}
\includegraphics[width=6.5cm, height=6.5cm]{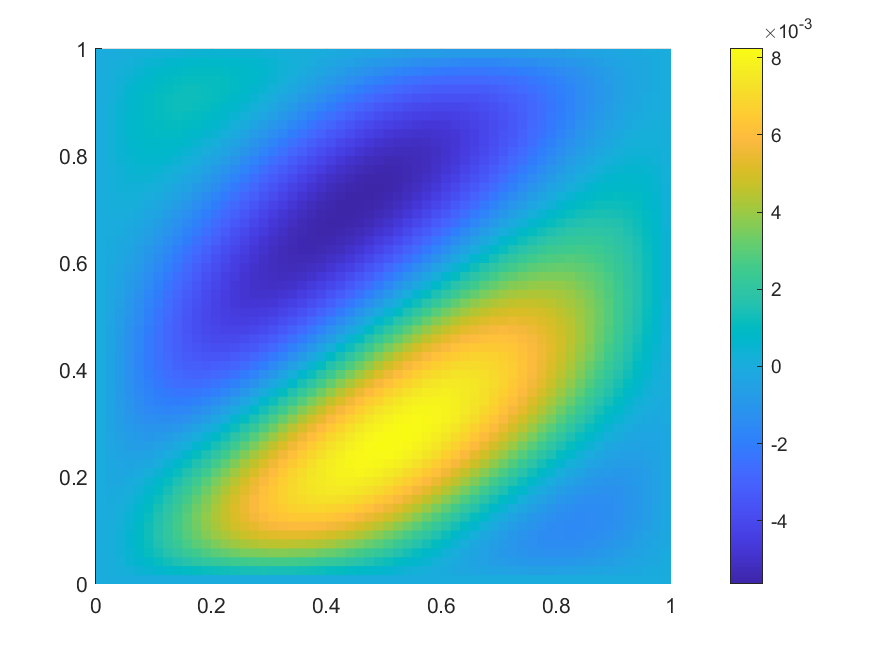}\includegraphics[width=6.5cm, height=6.5cm]{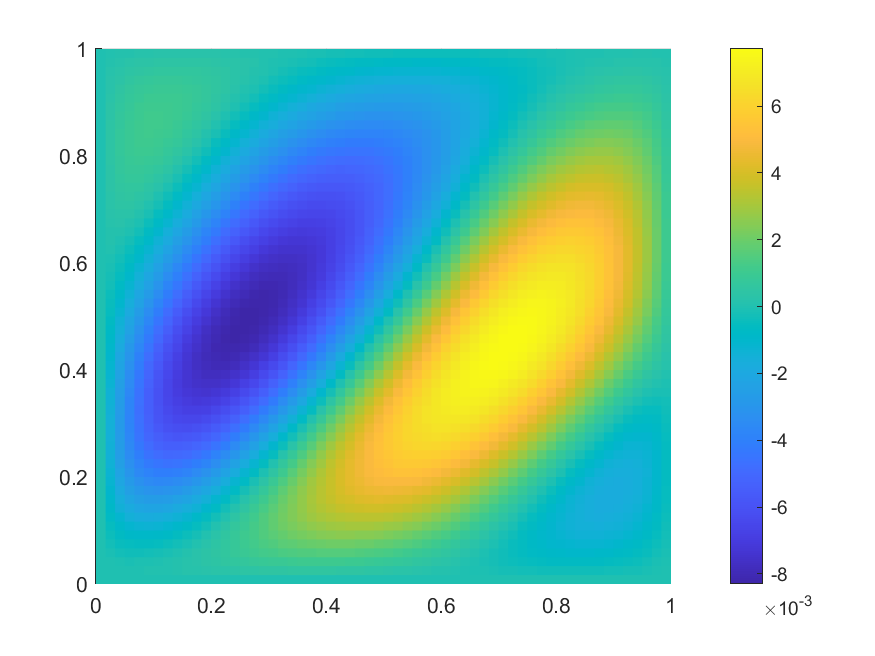}
\caption{Pointwise nodal errors in the displacement,   $|u_1-u_{1_h}|$ on right and $|u_2-u_{2_h}|$ on left. }
\label{Fig 2}
\end{center}
\end{figure}

{\bf Example 2:}  
 This example is devoted to confirm the QMC theoretical convergence results when $\mu$ is random and  $\lambda$ is constant. More precisely, $\lambda=1$ and  
\[
\mu(\bsx,\bsy)= \frac{1}{10}\Big(1 + \sum_{j=1}^\infty
 y_j \psi_j\Big),
\quad {\rm with}~~ \psi_j=\frac{1}{j^2} \sin (j \pi x_1) \sin((2j-1)\pi x_2),\]
and for $y_j\in [-1/2,1/2].$ Since  $\|\psi_j\|\le {1}/{(2j^2)},$ $\sum_{j=1}^\infty \|\psi_j\|^p$ is convergent for $p>1/2$ and  $\sum_{j=s_1+1}^\infty \|\psi_j\|\le C s_1^{-1}\,.$ Thus,  \eqref{ass A1} and \eqref{ass A3} are satisfied when $p>1/2$ and $\varrho_1=1/2,$ respectively. We discretize on the physical domain using the quadratic FEM. Therefore, according to  Theorem \ref{main results}, we expect the truncated QMC Galerkin finite element error to be of order $\calO(s_1^{-1}+\log(s_1)(N_1^{-2}+h^3))$. The  appearance of the logarithmic factor $\log(s_1)$ in front of $N_1^{-2}$ and $h^3$ is  due to the facts that $\sum_{j=1}^{s_1} \|\psi_j\|^{1/2}\le C \log(s_1)$ and that $\sum_{j=1}^{s_1} \|\nabla \psi_j\|\le C \log(s_1)$, respectively. For measuring the error, and since the exact solution is unknown, we rely on the reference solution $\Xi_{\vu_h^*}$ which is computed using  $s_1=256$, $J=128,$   and $1024$ QMC points.   Hence,  by ignoring the logarithmic factor $\log(s_1)$, and in the absence of the truncated series error,  we anticipate   $\calO(N_1^{-2})$-rates of convergence for $N_1\le J^{\frac{3}{2}}$, with $N_1 \ll 1024$. This is  illustrated numerically in Table \ref{tab:Example2} and graphically in Figure \ref{fig:errN example2} for different values of $N_1$, and for fixed $s_1=256$ and $J=128$, with $\calL=\calL_{\bf 1}$ and $\vf=(2x_1+10,x_2-3).$ 
Note that the middle column of Table~\ref{tab:Example2} displays
$|\Xi_{\vu^*_h}-\Xi_{{\vu_h}, Q}|$  where $Q$ is a quadrature for $\calL = \calL_1$ with $N_2=1$ and $N_1$ varying, see \eqref{eq:QMCInt F}.

\begin{table}[ht]
\centering
\begin{tabular}{|c|cc|}
\hline
$N_1$  &   { $|\Xi_{\vu^*_h}-\Xi_{{\vu_h}, Q}|$}~~ & CR\\
\hline
 8  &  3.1045e-01 &        \\
 16 &  6.1906e-02 & 2.3262 \\
 32 &  1.5191e-02 & 2.0269 \\
 64 &  4.3387e-03 & 1.8079 \\
128 &  9.3546e-04 & 2.2135 \\
256 &  2.5232e-04 & 1.8904 \\
\hline
\end{tabular}
\caption{Example 2, errors and convergence rates for different values of $N_1$.}
\label{tab:Example2}
\end{table}

\begin{figure}[ht]
    \centering
    \includegraphics[scale=0.5]{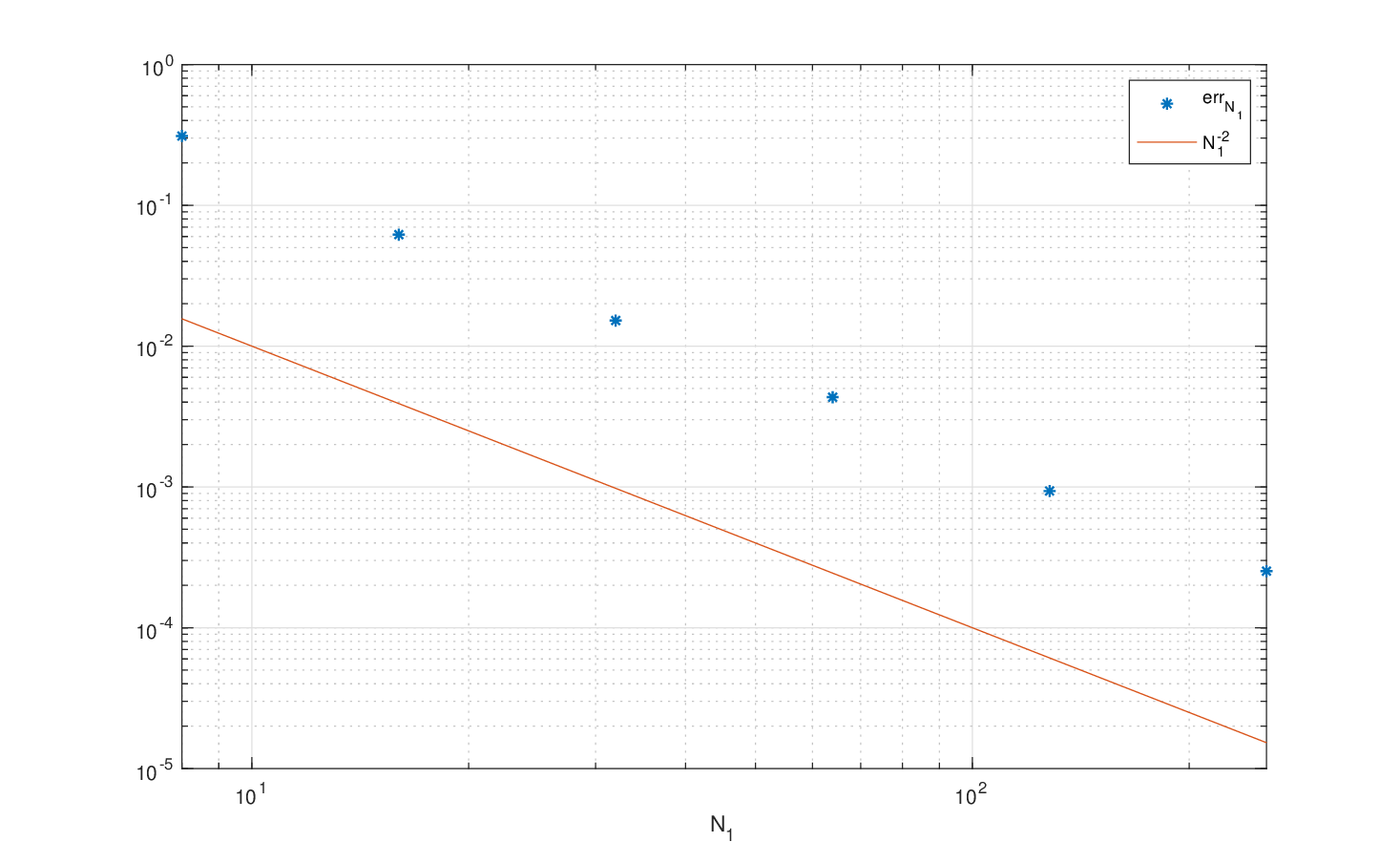}
    \caption{Numerical errors (err$_{N_1}$) vs. $N_1^{-2}$ for Example 2}
    \label{fig:errN example2}
\end{figure}

{\bf Example 3:}  In this example, we focus on the randomness in $\lambda$ while $\mu=1$. We choose 
\[
\lambda(\bsx,\bsz)= 1 + \sum_{j=1}^\infty
 \frac{z_j}{j^2} \sin (j \pi x_1) \sin((2j-1)\pi x_2),
\quad z_j \in [-1/2,1/2].
\]
By arguing as in the preceding example, based on  Theorem \ref{main results}, we fix  $s_2 =256$, $J=128$ and $r=2,$  then the QMC Galerkin finite element error is expected to be of order $\calO(N_2^{-2})$ whenever $N_2\le J$,  where the logarithmic factor $\log(s_2)$ is ignored.  We rely again on the reference solution $\Xi_{\vu_h^*}$, which is computed as in the previous example, in measuring the errors, and consequently, the convergence rates. As expected,  an  $\calO(N_2^{-2})$  convergence rate is illustrated tabularly and graphically for different values of $N_2$ in Table  \ref{tab:Example3}  and Figure \ref{fig:errN example3}, respectively, for fixed  $s_2=J=256,$ with $\calL=\calL_{\bf 1}$ and $\vf=(2x_1+10,x_2-3).$  
Note that the middle column of Table~\ref{tab:Example2} displays
$|\Xi_{\vu^*_h}-\Xi_{{\vu_h}, Q}|$ where $Q$ is a quadrature for $\calL = \calL_1$ with $N_1=1$ and $N_2$ varying, see \eqref{eq:QMCInt F}.
\begin{table}[ht]
    \centering
    \begin{tabular}{|c|cc|}
\hline    
$N_2$  &   { $|\Xi_{\vu^*_h}-\Xi_{{\vu_h}, Q}|$}~~ & CR\\
\hline
 8  &  7.5520e-04 &  \\
 16 &  2.0011e-04 & 1.9161 \\
 32 &  4.5223e-05 & 2.1457 \\
 64 &  1.1630e-05 & 1.9592 \\
128 &  2.7057e-06 & 2.1038 \\
256 &  6.5202e-07 &  2.0530 \\
\hline
\end{tabular}
\caption{Example 3, errors and convergence rates for different values of $N_2$.}
\label{tab:Example3}
\end{table}

\begin{figure}[ht]
    \centering
    \includegraphics[scale=0.5]{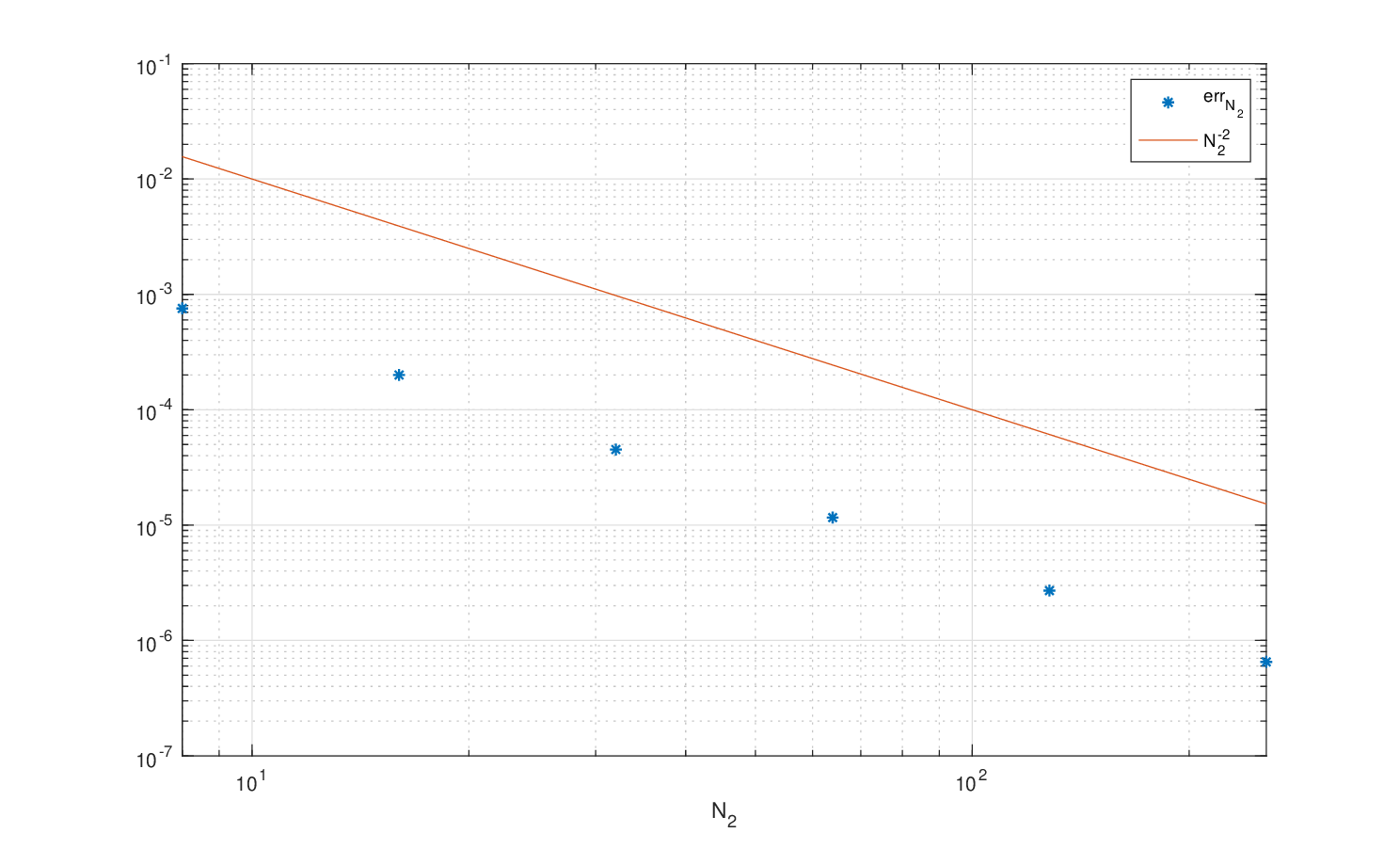}
    \caption{Numerical errors (err$_{N_2}$) vs. $N_2^{-2}$ for Example 3}
    \label{fig:errN example3}
\end{figure}

{\bf Example 4:} The aim behind this example is to illustrate numerically the achieved direct $(s_1+s_2)$-dimensional QMC  (second part of Theorem \ref{main results} or Theorem~\ref{prop:qmc6.3}) and the  QMC sparse grid  (third part of  Theorem \ref{main results} or Theorem ~\ref{qmc sparse grid}) convergence results.  As before,  $\calL=\calL_{\bf 1}$ and we set the body force $\vf=(2x_1+10,x_2-3)$ but now both coefficients $\lambda$ and $\mu$ are random. We choose 
\begin{align*}
    \mu(\bsx,\bsy)&= 1 + \sum_{j=1}^\infty
 \frac{y_j}{j^2} \sin (j \pi x_1) \sin((2j-1)\pi x_2),
\quad y_j \in [-1/2,1/2],\\
\lambda(\bsx,\bsz)&= 1 + \sum_{j=1}^\infty
 \frac{z_j}{j^2} \sin (j \pi x_1) \sin((2j-1)\pi x_2),
\quad z_j \in [-1/2,1/2],
\end{align*}
and so,  $p=q=1/2$ (note that strictly speaking we have $p=q=1/2+\epsilon$ for an arbitrary $\epsilon > 0$; in order to simplify the computation we ignore this technicality in the following). We fix the truncation degree $s_1 = s_2 =  256$, the spatial mesh element size $J=128,$ and the degree of the Galerkin FEM $r=2.$ The reference solution $\Xi_{\vu_h^*}$ is generated using a full grid of $2048 \times 2048$ (that is, $b=2$ and $m_1=m_2=11$) high-order QMC points (generated by a Python package in \cite{Gantner2014}). Note that the PDE solvers can be run in parallel for distinct QMC points.  To speed up the computation, finite element PDE solvers based on examples in the FEniCS package \cite{LangtangenLogg2017} are used on the high-performance computing platform Katana \cite{Katana} provided by UNSW, Sydney. The Python code used in the numerical experiments together with the PBS scripts is available at \url{https://github.com/qlegia/Elasticity-HigherOrder-QMC}.

The one family of QMC sparse grid algorithm (discussed in Theorem~\ref{prop:qmc6.3}) is implemented to compute $\Xi_{\vu_h,Q_N}$ where $N$ is the total number of QMC points. For different values of $N$, the errors between approximation $\Xi_{\vu_h,Q_N}$ and the reference solution are given in the second column of Table~\ref{tab:example 4 Theorem 6.3}. The expected $O(N^{-2})$-rates of convergence is illustrated numerically in the third column. 

The combined QMC sparse grid algorithm \eqref{eq: truncate infinite sums}  (with ${\bf N}^{L-k,k}=(2^{L-k},2^k)$, that is, $\vartheta = 2$) is implemented to compute $\Xi_{\vu_h,Q_L}$. The errors between approximation $\Xi_{\vu_h,Q_L}$ and the reference solution for different values of $L$  are given in the second column of Table~\ref{tab:sparse grid results}. The fourth column of the table gives the QMC sparse grid upper error bounds $(\log\,M) M^{-1}$ predicted by Theorem ~\ref{qmc sparse grid} (where the constant $C$ in the error bound  is ignored). 
\begin{table}[]
    \centering
    \begin{tabular}{|c|c|c|}
    \hline   
 $N$  &  $|\Xi_{\vu_h^*}-\Xi_{\vu_h,Q_N}|$ & CR \\
 \hline
  256  & 6.4537e-07  & \\
  512  & 1.5738e-07  & 2.0359\\
 1024  & 2.8466e-08  & 2.4670\\
 2048  & 9.8881e-09  & 1.5255\\
 4096  & 2.2407e-09  & 2.1417\\
 8192  & 4.9127e-10  & 2.1894 \\
 \hline
    \end{tabular}
    \caption{Example 4 using one family of QMC rule as in Theorem~\ref{prop:qmc6.3}}
    \label{tab:example 4 Theorem 6.3}
\end{table}

\begin{table}[]
    \centering
    \begin{tabular}{|l|c|c|c|c|c|}
    \hline
    $L$  &  $M=(L-1)2^L$ & $|\Xi_{\vu_h,Q_L} - \Xi_{\vu_h^*}|$~~~ & $(\log \, M) M^{-1}$ \\
    \hline
    9  & 4096   & 3.5687e-06 &2.0307e-03\\
    10 & 9216   & 3.4606e-06 &9.9053e-04\\
    11 & 20480  & 7.4782e-06 &4.8473e-04\\
    12 & 45056  & 1.1456e-06 &2.3783e-04\\
    13 & 98304  & 6.8586e-08 &1.1694e-04\\
    14 & 212992 & 1.4453e-07 &5.7603e-05\\
    15 & 458752 & 5.4749e-08 &2.8417e-05\\
    \hline
    \end{tabular}
    \caption{Example 4, numerical and theoretical error results for QMC sparse grid algorithm.}
    \label{tab:sparse grid results}
\end{table}

\end{document}